\documentclass{amsart}

\usepackage{amsthm} 
\usepackage{color}

\usepackage{todonotes}
\usepackage{mathtools}

\usepackage{graphicx,epsfig}
\usepackage{epstopdf}
\usepackage{amsmath, amssymb, latexsym, euscript,amscd}
\usepackage{url}
\usepackage[all]{xy}
\usepackage{psfrag}
\usepackage[pagebackref=true]{hyperref}
\usepackage{mathrsfs}
\usepackage{tikz}
\usetikzlibrary{shapes.geometric}
\usepackage{subcaption}
\usepackage{multicol}
\usepackage{extarrows}

\usepackage{fullpage}

\usetikzlibrary{shapes.geometric}
\usetikzlibrary{positioning} 
\usepackage[alphabetic]{amsrefs}

\newcounter{notes}%

\definecolor{darkgreen}{rgb}{0.0, 0.5, 0.0}

\newtheorem{theorem}{Theorem}[section]
\newtheorem{lemma}[theorem]{Lemma}
\newtheorem{corollary}[theorem]{Corollary} 
\newtheorem{definition}[theorem]{Definition} 
\newtheorem{proposition}[theorem]{Proposition}
\newtheorem{remark}[theorem]{Remark}

\def\smallskip{\vspace{.15cm}}
\def\medskip{\vspace{.3cm}}
\def\text{\mbox}
\def\rh2{{\mathbb R}{\mathbb H}^2}
\def\ch2{{\mathbb C}{\mathbb H}^2}

\def\QQ{{\mathbb Q}}

\def\EE{{\mathbb E}}
\def\ZZ{{\mathbb Z}}

\def\FF{{\mathbb F}}

\def\P{{\mathbb P}}

\def\RP2{{\mathbb{RP}}^2}
\def\RP3{{\mathbb{RP}}^3}
\def\RP{{\mathbb{RP}}}
\def\tri{{\mathcal T}}

\def\Id{\operatorname{Id}}

\def\SL{\operatorname{SL}}
\def\Id{\operatorname{Id}}

\def\End{\operatorname{End}}

\def\GL{\operatorname{GL}}
\def\int{\operatorname{int}}

\def\H2R{{\mathbb H}^2\times {\mathbb R}}

\def\C2{\operatorname{C^2}}

\def\Fix{\mathrm{Fix}}

\definecolor{back}{RGB}{255,255,255}
\definecolor{fore}{RGB}{0,0,0}
\definecolor{title}{RGB}{255,0,90}

\definecolor{green}{rgb}{0.0, 0.5, 0.0}
\definecolor{purple}{rgb}{0.5, 0.0, 0.5}
\definecolor{bluegreen}{rgb}{0.0,0.5, 0.5}
\definecolor{orange}{rgb}{1,0.5, 0.1}
\definecolor{redgreen}{rgb}{0.5, 0.5, 0.0}

\def\red{\color{red}}
\def\green{\color{green}}

\def\red{\color{red}}
\def\green{\color{green}}

\def\g2{{\green 2}}

\newcommand{\bv}{\left[\begin{array}{c}}
\newcommand{\ev}{\end{array}\right]}
\newcommand{\bbmat}{\begin{bmatrix}} 
\newcommand{\ebmat}{\end{bmatrix}}
\newcommand{\bmat}{\begin{matrix}}
\newcommand{\emat}{\end{matrix}}
\newcommand{\bpmat}{\begin{pmatrix}} 
\newcommand{\epmat}{\end{pmatrix}}

\begin{document}
\title{On Wonderful compactifications of $\SL(2, \FF)$\\ for non-Archimedean local fields $\FF$} 
\author{Corina Ciobotaru}
\thanks{corina.ciobotaru@gmail.com}
\date{January 2, 2023}

\begin{abstract}
We compute the wonderful compactification of symmetric varieties of $\SL(2, \FF)$, where $\FF$ is a finite field-extension of $\QQ_p$ with $p\neq 2$, that comes from either an abstract or $\FF$-involutions of $\SL(2, \FF)$. For each of those wonderful compactifications we find the $\SL(2, \FF)$-stabilizers of the accumulation points of the corresponding symmetric varieties and compare them to the Chabauty limits found in \cite{CL_2}.
\end{abstract}

\maketitle 

\section{Introduction}

The wonderful compactifications emerged from specific problems in enumerative geometry. In laymen's terms a wonderful compactification is a way of compactifying a symmetric variety $G/H$ via an irreducible linear representation of the group $G$, where $G$ is an algebraic group and $H \leq G$ is the fixed point set of an involutorial automorphism of $G$.  This procedure will provide remarkable properties of the $G$-closed orbits of that compactification.  In principal, the wonderful compactification of $G/H$ does not depend on the chosen irreducible linear representation of $G$, and a Borel subgroup of $G$ will only have one open and dense orbit in $G/H$. The study of  wonderful compactifications was initiated by DeConcini and Procesi \cite{DCP} in the case of a complex semi-simple group $G$ of adjoint type (i.e. a semi-simple group with trivial center).  Since then wonderful compactifications have been extensively studied in algebraic geometry and have important applications in many fields of Mathematics (see the Introduction of \cite{EJ}). If one replaces the complex numbers by any (not necessarily algebraically closed) field $k$ of characteristic $\neq 2$, the results of \cite{DCP}  were further expanded by DeConcini and Springer in \cite{DCS} to any adjoint semi-simple group $G$ defined over $k$ and $H$ the fixed point group of an involutorial automorphism of $G$ defined over $k$.

In this article we focus on the wonderful compactifications of symmetric varieties of $\SL(2, \FF)$, where $\FF$ a finite field-extension of $\QQ_p$ with $p\neq 2$, that come from various abstract and $\FF$-involutions of $\SL(2, \FF)$. Notice that $\SL(2,\FF)$ is not of adjoint type. In particular, we are interested in computing the $\SL(2, \FF)$-stabilizers of the accumulation points in the  wonderful compactifications of those symmetric varieties and we compare them to the Chabauty limits found in \cite{CL_2}. Although some of the techniques used in this article to build the wonderful compactifications are similar to the well established literature (e.g. \cite{EJ}), we will make extensive use of some results from \cite{Beun}, the Bruhat--Tits tree and its visual boundary associated with $\SL(2, \FF)$, which contrast with the methods employed in \cite{DCS}.

\medskip
Let us first fix the notation. Throughout this article we restrict to a prime $p\neq 2$. Let $\FF$ be a finite field-extension of $\QQ_p$ and $\EE$ be any quadratic extension of $\FF$. Let $k_\FF, k_\EE$ be the residue fields of $\FF,\EE$, respectively,  and $\omega_\FF, \omega_\EE$ be  uniformizers of $\FF,\EE$, respectively. Recall $k_\FF^{*}/(k_\FF^{*})^{2} =\{1, S\}$, for some non-square $S \in k_\FF^{*}$. Then $\FF^{*}/(\FF^{*})^{2} =\{1, \omega_\FF, S, S\omega_\FF\}$ (\cite[Corollaries to Theorems 3 and 4]{Serre} or \cite[page 41, Section 12]{Sally}). If $\EE$ is \emph{ramified} then $\EE = \FF(\sqrt{\omega_\FF})$, or $\EE= \FF(\sqrt{S \omega_\FF})$ (where $\omega_\FF \neq \omega_\EE$), or if $\EE$ is \emph{unramified} then $\EE=\FF(\sqrt{S})$ (where $\omega_\FF = \omega_\EE$).    We choose the unique valuation $\vert \cdot \vert_\EE$ on $\EE$  that extends the given valuation $\vert \cdot \vert_\FF$ on $\FF$. Choose $\alpha \in \{\sqrt{\omega_\FF}, \sqrt{S}, \sqrt{S\omega_\FF}\}$ and so
$\EE=\FF(\alpha)$. Notice each element $x \in\EE$ can be uniquely written as $x = a+b \alpha$, with $a,b \in \FF$. For the ramified extensions we can consider $\omega_\EE^2 = \omega_\FF$. Let $\mathcal{O}_\FF:=\{ x \in \FF \; \vert \; \vert x \vert_\FF \leq 1 \}$ denote the ring of integers of $\FF$,  then $\mathcal{O}_\FF$ is compact and open in $\FF$. For $\FF=\QQ_p$ we have $\mathcal{O}_{\QQ_p}=\ZZ_p$, $\omega_{\QQ_p}=p$, $k_{\QQ_p}=\FF_p$, $\FF_p^{*}/(\FF_p^{*})^{2} =\{1, S_p\}$.

We denote by $\tri_\FF$ the Bruhat--Tits tree for $\SL(2,\FF)$ whose vertices  are equivalence classes of $\mathcal{O}_\FF$-lattices in $\FF^2$ (for its construction see \cite{Serre}). The tree $\tri_\FF$ is a regular, infinite tree with valence $\vert k_\FF\vert +1$ at every vertex. The boundary at infinity $\partial \tri_{\FF}$ of $\tri_{\FF}$ is the projective space $P^1(\FF) \cong \FF \cup \{\infty\}$. Moreover, the end $\infty \in \partial \tri_{\FF}$ corresponds to the vector  $\big[ \begin{smallmatrix}  0 \\1 \end{smallmatrix} \big] \in P^1(\FF)$. The rest of the ends $\xi \in \partial \tri_\FF$ correspond to the vectors $\big[ \begin{smallmatrix} 1 \\x \end{smallmatrix} \big] \in P^1(\FF) $, where $x \in \FF$. To give a concrete example,  the Bruhat--Tits tree of $\SL(2, \QQ_p)$ is the $p+1=\vert \FF_p \vert+1$-regular tree. The boundary at infinity $\partial \tri_{\QQ_p}$ of $\tri_{\QQ_p}$ is the projective space $P^1(\QQ_p)=\QQ_p \cup \{\infty\}$.

Let $\rho : \SL(2,\FF) \to \GL(\FF^{2})$ be the representation of $\SL(2,\FF)$ over $\FF$ given by $g \in \SL(2,\FF) \mapsto \rho(g):=g$.
One can easily prove $\rho$ is an irreducible and continuous representation of $\SL(2,\FF)$ with respect to the compact-open topologies of the locally compact groups $\SL(2,\FF)$ and $\GL(\FF^{2})$.  By the known theory $\rho$ has a dominant weight (see \cite[Theorem 2.5]{Tits}). 

For an abstract involution $\theta : \SL(2,\FF)  \to \SL(2,\FF)$, i.e. \textit{$\theta$ is an abstract automorphism of $\SL(2,\FF)$ with $\theta^2 =\Id$}, we consider: 
\begin{enumerate}
\item
the  action of $\SL(2,\FF)$ on $\End(\FF^2)$ given by $g\cdot_{\theta} D := \rho(g) D \rho(\theta(g^{-1}))$, for every $g \in \SL(2,\FF)$ and $D \in \End(\FF^2)$,
\item
 the continuous map $\psi_{\theta} : \SL(2,\FF) \to \P(\End(\FF^2))$ given by  $g \mapsto \psi_{\theta}(g):=[\rho(g)\rho(\theta(g^{-1}))]=[\rho(g \theta(g^{-1}))]$,
\item
and the fixed point group of $\theta$ denoted by $H_{\theta}:= \{g \in \SL(2,\FF) \; \vert \; \theta(g)=g\}$.
\end{enumerate}

\begin{definition}
We take $X_\theta:=\overline{\psi_{\theta}(\SL(2,\FF))} = \overline{[\SL(2,\FF)\cdot_{\theta}\Id_{\FF^{2}}]}  \subset \P(\End(\FF^2))$ and, by abuse of terminology, call it \textit{the wonderful compactification of $\SL(2,\FF)/H_{\theta}$ with respect to the involution $\theta$}.
\end{definition}

We will see in the proofs below that a Borel subgroup of $\SL(2,\FF)$ might have more open orbits on  $\SL(2,\FF)/H_{\theta}$, that are pairwise disjoint and not dense in $\SL(2,\FF)/H_{\theta}$. This is because $\SL(2,\FF)$ is not of adjoint type. By abuse of notation, for any $g \in \SL(2,\FF)$ and any $[D] \in X_{\theta} \subset \P(\End(\FF^2))$ we use $g\cdot_{\theta} [D]$ to mean $[\rho(g) D\rho(\theta(g^{-1}))] \in  X_{\theta}$ for some, thus any, representative $D \in \End(\FF^2)$ of $[D] \in \P(\End(\FF^2))$.

By the known theory, those wonderful compactifications do not depend on the irreducible representation $\rho$ of $\SL(2,\FF)$ (for a proof see \cite[Section 7]{Ana} or \cite[Proposition 3.10]{DCS}).

The three main results of the paper are as follows. They are proved in Sections \ref{sec::diag}, \ref{sec::all_orbits_E}, \ref{sec::all_orbits}, respectively, where  the reader can refer for the notation.

\begin{theorem}
\label{thm::all_orbits_diag}
Consider the involution $\theta(x,y):=(y,x)$ of  $\SL(2,\FF) \times \SL(2,\FF)$ where $H_{\theta}:=Diag(G)=\SL(2,\FF)$ is the fixed point group, and $\SL(2,\FF) \times \SL(2,\FF)$ acts on $\End(\FF^2)$ by $(g_1,g_2)\cdot_{\theta}A=\rho(g_1) A \rho(g_2^{-1})$. 
Then there are  $2$ orbits of $\SL(2,\FF) \times \SL(2,\FF)$ in $X_{\theta}$ with respect to the action $\cdot_\theta $:
\begin{enumerate}
\item
the open $\SL(2,\FF) \times \SL(2,\FF)$-orbit of $\Id_{\FF^2}$, whose $\SL(2,\FF) \times \SL(2,\FF)$-stabilizer is $Diag(\SL(2,\FF) \times \SL(2,\FF))= \SL(2,\FF)$  
\item
the closed $\SL(2,\FF) \times \SL(2,\FF)$-orbit of $[1,0;0,0]  \in \P(\End(\FF^2))$, whose $\SL(2,\FF) \times \SL(2,\FF)$-stabilizer is $B^{+} \times B^{-}$. 
\end{enumerate}
\end{theorem}

\begin{theorem}
\label{thm::all_orbits_E}
 Let $\FF$ be a finite field-extension of $\QQ_p$ and $\EE=\FF(\alpha)$ be any quadratic extension of $\FF$. Take $\theta$ to be the involution of $\SL(2,\EE)$ induced by the `conjugation' in $\EE$ with respect to $\FF$,  $\theta( a + \alpha b) := a - \alpha b$, for any $a,b\in \FF$.  There are $2$ orbits of $\SL(2,\EE)$ in $X_{\theta}$  with respect to the action $\cdot_\theta$:
\begin{enumerate}
\item
the open $\SL(2,\EE)$-orbit of $\Id_{\EE^2}$, whose  $\SL(2,\EE)$-stabilizer is $\SL(2,\FF)$
\item
the closed $\SL(2,\EE)$-orbit of $[0,  1; 0, 0] \in \P(\End(\EE^2))$, whose $\SL(2,\EE)$-stabilizer is  the subgroup 
$\{\left( \begin{smallmatrix} a-\alpha b& z \\  0 &  a+\alpha b \end{smallmatrix} \right) \; \vert \;  a,b \in\FF \text{ with }a^2 - \alpha^2 b^{2}=1,  z \in \EE \} \leq B_{\EE}^{+}.$ 
\end{enumerate}
\end{theorem}

\begin{theorem}
\label{thm::all_orbits}
Let $ m \in \FF^{*}/(\FF^{*})^{2} =\{1, \omega_\FF, S, S\omega_\FF\}$. Consider $A_m:=\left( \begin{smallmatrix} 0 &1 \\ m &0 \end{smallmatrix} \right)$ and the corresponding inner $\FF$-involution of $\SL(2,\FF)$ given by  $\sigma_m:=\iota_{A_m}$ (see Section \ref{sec::all_orbits}).
There are $2$ orbits of $\SL(2,\FF)$ in $X_{\sigma_m}$ with respect to the action $\cdot_{\sigma_m}$: 
\begin{enumerate}
\item
the open  $\SL(2,\FF)$-orbit of $\Id_{\FF^2}$, whose $\SL(2,\FF)$-stabilizer is $H_{\sigma_m}$
\item
the closed $\SL(2,\FF)$-orbit of $[1,  0; 0, 0] \in \P(\End(\FF^2))$,  whose $\SL(2,\FF)$-stabilizer is  the subgroup $\{ \mu _2 \cdot \left( \begin{smallmatrix} 1& b \\ 0 &1  \end{smallmatrix} \right)\; |\; b \in \FF  \}$ of the Borel  $B^{+}\leq \SL(2,\FF)$, where $\mu _2$ is the group of $2^{nd}$ roots of unity in  $\FF$.
\end{enumerate}
\end{theorem}

Comparing the stabilizers of the accumulation points $[0, 1; 0, 0]$ and $[1, 0; 0, 0]$ computed in Theorems \ref{thm::all_orbits_E} and \ref{thm::all_orbits}, respectively, with the Chabauty limits computed in \cite{CL_2} Theorems 1.4 and 1.6, respectively, the reader can notice we found the same subgroups. The general case of $\SL(n,\FF)$ will be treated in a future research project.


\subsection*{Acknowledgements}  Ciobotaru is supported by the European Union’s Horizon 2020 research and innovation programme under the Marie Sklodowska-Curie grant agreement No 754513 and The Aarhus University Research Foundation. She was partially supported by The Mathematisches Forschungsinstitut Oberwolfach (MFO, Oberwolfach Research Institute for Mathematics). She would like to thank those two institutions for the perfect working conditions they provide, and to Linus Kramer and Maneesh Thakur for very helpful discussions regarding reference \cite{Tits}. As well, Ciobotaru thanks Arielle Leitner for the wonderful collaboration on our joint paper \cite{CL_2} that gave rise to this article, for reading it and providing useful comments.

\section{The wonderful compactification of $\SL(2,\FF)$}
\label{sec::diag}

First recall  the usual wonderful compactification associated with the involution $\theta(x,y):=(y,x)$ of  $G:=\SL(2,\FF) \times \SL(2,\FF)$, having $Diag(G)=\SL(2,\FF)$ as a fixed point group of the involution $\theta$.
We compactify $\SL(2,\FF)=G/Diag(G)$  in the following way. Consider $\rho(\SL(2,\FF)) \leq \GL(\FF^2) \subseteq \End(\FF^2)$ and $\SL(2,\FF) \times \SL(2,\FF)$ acting on $\End(\FF^2)$ by $(g_1,g_2)\cdot_{\theta}A=\rho(g_1) A \rho(g_2^{-1})$, where $(g_1,g_2) \in \SL(2,\FF) \times \SL(2,\FF)$ and $A \in \End(\FF^2)$.  Then the stabilizer in $\SL(2,\FF) \times \SL(2,\FF)$ of $\Id_{\FF^2}$ is exactly $Diag(G)=\SL(2,\FF)$, and $\rho(\SL(2,\FF))= (\SL(2,\FF) \times \SL(2,\FF))\cdot_{\theta}\Id_{\FF^{2}}$. 

Next we take the map
$$\psi_{\theta} : \SL(2,\FF) \times \SL(2,\FF) \to \P(\End(\FF^2)), \; \text{ given by } (g_1,g_2) \mapsto \psi_{\theta}(g_1,g_2):=[\rho(g_1g_2^{-1})]$$
and  we want to understand the closure
 $$ X_{\theta}:=\overline{\psi_{\theta}(\SL(2,\FF) \times \SL(2,\FF))}=\overline{[(\SL(2,\FF) \times \SL(2,\FF))\cdot_{\theta}\Id_{\FF^{2}}]} =  \overline{[ \rho(\SL(2,\FF))]}\text{ in }\P(\End(\FF^2))$$
 called \textit{the wonderful compactification of $\SL(2,\FF)=G/Diag(G)$ with respect to the involution $\theta$}, where $\P(\End(\FF^2))$ is endowed with the usual topology.

Denote 
\begin{equation}
\label{equ::notations}
\begin{split}
& T:= \left \{ \begin{pmatrix} x & 0 \\ 0 &x^{-1}  \end{pmatrix} | \; x \in \FF^* \right\}  \qquad  U^{+}:= \left \{ \begin{pmatrix} 1 & y \\ 0 &1  \end{pmatrix} |  \; y \in \FF \right\} \qquad  U^{-}:= \left \{ \begin{pmatrix} 1 & 0 \\ y &1  \end{pmatrix} |  \; y \in \FF \right\} \\
&B^{+}:= \left \{ \begin{pmatrix} x & y \\ 0 &x^{-1}  \end{pmatrix} |  \; y \in \FF, x \in \FF^*  \right\} \qquad B^{-}= \left \{ \begin{pmatrix} x & 0 \\ y &x^{-1}  \end{pmatrix} |  \; y \in \FF, x \in \FF^*  \right\}, \\
\end{split}
\end{equation}
where $T$ is the maximal tori of $\SL(2,\FF)$, $U^{+}$ is the unipotent radical of the Borel subgroup $B^{+}$, and $U^{-}$ is the unipotent radical of the opposite Borel  $B^{-}$. By the definition of $X_\theta$:

\begin{lemma}
\label{lem::dense_orbit}
The  $\SL(2,\FF) \times \SL(2,\FF)$ orbit of $\Id_{\FF^2}$ in $\P(\End(\FF^2))$ with respect to the action $\cdot_{\theta}$  is dense in $X_\theta$.
\end{lemma}


In order to prove  Theorem \ref{thm::all_orbits_diag} we need the following notation. Let $$\P_{0}:=\{x =[1,  x_2; x_3, x_4] \in \P(\End(\FF^2)) \}$$ and notice this is open in $\P(\End(\FF^2))$. Denote $X_{\theta, 0}:= \P_{0} \cap X_{\theta}$ and notice this is again open in $X_\theta$. The set  $X_{\theta, 0}$ is called the big cell in  $X_\theta$.

Recall that $B^{+}=T U^{+}$ as well as the following decomposition in $\SL(2,\FF)$
\begin{equation}
\label{equ::decomp_unipotent_rad}
\begin{pmatrix}a&b\\c&d\end{pmatrix} =  \begin{pmatrix}1&0\\c/a&1\end{pmatrix}  \begin{pmatrix}a&0\\0&a^{-1}\end{pmatrix}  \begin{pmatrix}1&b/a\\0&1\end{pmatrix}, \text{ for } a\neq 0.
\end{equation}
From here it is easy to verify the following facts. 

\begin{lemma}
\label{lem::dense_utu}
The set  $U^{-}TU^{+}$ is dense in $\SL(2,\FF)$, and in particular $[(U^{-}T \times U^{+})\cdot_{\theta}\Id_{\FF^{2}}]$ is dense in $\psi_{\theta}(\SL(2,\FF) \times \SL(2,\FF)) = [\rho(\SL(2,\FF))]$, and so in $\overline{ [\rho(\SL(2,\FF))]}$.
\end{lemma}

\begin{proof}
The first part of the lemma follows from an easy approximation of matrices having $a=0$ using the decomposition given in (\ref{equ::decomp_unipotent_rad}). The rest of the lemma follows directly from the density of $U^{-}TU^{+}$ in $\SL(2,\FF)$.
\end{proof}

\begin{lemma}
\label{lem::X0_int_G}
We have that 
$$X_{\theta,0} \cap [\rho(\SL(2,\FF))]= [(U^{-}T \times U^{+})\cdot_{\theta}\Id_{\FF^{2}}]$$
$$X_{\theta}= \bigcup\limits_{(g_1,g_2) \in \SL(2,\FF) \times \SL(2,\FF)} (g_1,g_2)\cdot_{\theta}X_{\theta,0}$$
$$X_{\theta,0} = \overline{[ (U^{-} T \times U^{+})\cdot_{\theta}\Id_{\FF^{2}}]} \cap \P_{0}.$$
\end{lemma}
\begin{proof}
By the decomposition \ref{equ::decomp_unipotent_rad} it is clear that 
$$ [(U^{-}T \times U^{+})\cdot_{\theta}\Id_{\FF^{2}}] = X_{\theta,0} \cap [\rho(\SL(2,\FF))]=  \P_{0} \cap  [\rho(\SL(2,\FF))].$$
Since $\overline{[ (U^{-}T \times U^{+})\cdot_{\theta}\Id_{\FF^{2}}]} =  X_{\theta}$ we immediately have $X_{\theta,0} =\overline{[ (U^{-}T \times U^{+})\cdot_{\theta}\Id_{\FF^{2}}]} \cap  \P_{0}$, i.e. the closure of the set $[ (U^{-}T \times U^{+})\cdot_{\theta}\Id_{\FF^{2}}]$ in $\P_0$.

Let $C \in X_{\theta}$. As $C \in \P(\End(\FF^2))$ we have $C = [x_1, x_2; x_3,x_4]$. If $x_1 \neq 0$ then $C \in X_{\theta,0}$ and we are done. If $x_1 =0$, then by some easy calculations one can arrange for some $(g_1,g_2) \in \SL(2,\FF) \times \SL(2,\FF)$ such that $\rho(g_1) C \rho(g_2^{-1}) \in \P_0$. Then we are done.
\end{proof}

\begin{proof}[Proof of Theorem \ref{thm::all_orbits_diag}]

By Lemma \ref{lem::X0_int_G} it is enough to compute the set $(U^{-}T \times U^{+})\cdot_{\theta}\Id_{\FF^{2}}$, and then to precisely compute its closure $\overline{[ (U^{-}T \times U^{+})\cdot_{\theta}\Id_{\FF^{2}}]} \cap \P_{0} $. 
We have
\begin{equation}
\begin{split}
(U^{-}T \times U^{+})\cdot_{\theta}\Id_{\FF^{2}}&= \left \{ \begin{pmatrix} 1 & 0 \\ a & 1  \end{pmatrix}\begin{pmatrix} x & 0 \\ 0 &x^{-1}  \end{pmatrix}  \begin{pmatrix} 1 & -b \\ 0 & 1  \end{pmatrix}|\; x \in \FF^*, a,b\in \FF \right\}\\
& \left \{ \begin{pmatrix} x & -xb \\ ax & -axb+x^{-1}  \end{pmatrix}| \; x \in \FF^*, a,b\in \FF \right\} \Rightarrow \\
& \Rightarrow [ (U^{-}T \times U^{+})\cdot_{\theta}\Id_{\FF^{2}}] = \{[1,-b;a,-ab+x^{-2}] \; | \; x \in \FF^*, a,b\in \FF \}. \\
\end{split}
\end{equation}
Then the only accumulation points of the set $[ (U^{-}T \times U^{+})\cdot_{\theta}\Id_{\FF^{2}}]$  in $\P_0$ are the points of the form $[1,-b;a,-ab]$, for any $a,b\in \FF$. Moreover, it is very easy to see that the only accumulation point of the set $[T\cdot_{\theta}\Id_{\FF^{2}}]$ in $\P_0$ is just $[1,0;0,0]$.
Notice, an element $[1,-b;a,-ab]$ respresents a matrix in $\P(\End(\FF^2))$ of determinant zero. As well, $\{[1,-b;a,-ab]\; | \; a,b\in \FF \}$ is the $(U^{-}T \times U^{+})$-orbit of the element $[1,0;0,0] \in \P(\End(\FF^2))$.  From the above calculation and since the action $\cdot_{\theta}$ is continuous, it is immediate that the $\SL(2,\FF) \times \SL(2,\FF)$-orbit of $\Id_{\FF^{2}}$ is open. It is also clear that the $\SL(2,\FF) \times \SL(2,\FF)$-orbit of $[1,0;0,0]$ in $\P(\End(\FF^2))$ is closed and, by in easy computation, the stabilizer of $[1,0;0,0]$ in  $\SL(2,\FF) \times \SL(2,\FF)$ is exactly $B^{+} \times B^{-}$.
The theorem follows.
\end{proof}
 
\section{The wonderful compactification of $\SL(2, \EE)/\SL(2, \FF)$ for quadratic $\EE/\FF$}
\label{sec::all_orbits_E}

Let $\FF$ be a finite field-extension of $\QQ_p$ and $\EE=\FF(\alpha)$ be any quadratic extension of $\FF$.  Denote by $\theta :\EE  \to \EE$ the `conjugation' in $\EE$ with respect to $\FF$,  $\theta( a + \alpha b): = a - \alpha b$, for any $a,b\in \FF$. This map $\theta$ induces an abstract involution of $\SL(2,\EE)$ given by
$$\left( \begin{smallmatrix} x &y \\  z&t\end{smallmatrix} \right) \in \SL(2,\EE) \mapsto \theta(\left( \begin{smallmatrix} x &y \\  z&t\end{smallmatrix} \right)) := \left( \begin{smallmatrix} \theta(x) & \theta(y) \\ \theta(z)& \theta(t)\end{smallmatrix} \right)$$
whose fixed point group $H_{\theta}:=\{g \in \SL(2,\EE) \;  \vert \; \theta(g)=g\}$ is  $\SL(2,\FF)$.


The stabiliser in  $\SL(2, \FF)$, resp. $\SL(2, \EE)$, of the endpoint  $\big[ \begin{smallmatrix} 1 \\0 \end{smallmatrix}\big] $ is the Borel subgroup 
$$B_\FF^{+}:=\{ \left( \begin{smallmatrix} a &b \\ 0 &a^{-1}\end{smallmatrix} \right) \;|\; b\in \FF, a \in \FF^{\times} \}, \qquad \text{resp. } B_\EE^{+}:=\{ \left( \begin{smallmatrix} x &y \\ 0 &x^{-1}\end{smallmatrix} \right) \;|\;  y\in \EE, x \in \EE^{\times} \}. $$
The stabiliser in  $\SL(2,\FF)$, resp. $\SL(2, \EE)$, of the endpoint  $\big[ \begin{smallmatrix} 0 \\1 \end{smallmatrix}\big] $  is the opposite Borel subgroup 
$$B_\FF^{-}:=\{ \left( \begin{smallmatrix} a &0 \\  b&a^{-1}\end{smallmatrix} \right)\;|\;  b\in \FF, a \in F^{\times} \} , \qquad \text{resp. }  B_\EE^{-}:= \{ \left( \begin{smallmatrix} x &0\\ y &x^{-1}\end{smallmatrix} \right) \;|\;  y\in \EE, x \in E^{\times} \}. $$

Let  
\[ T_\EE:= \left \{ \begin{pmatrix} x & 0 \\ 0 &x^{-1}  \end{pmatrix} | \; x \in \EE^* \right\}  \qquad  U_{\EE}:= \left \{ \begin{pmatrix} 1 & y \\ 0 &1  \end{pmatrix} |  \; y \in \EE \right\} \qquad  U^{-}_{\EE}:= \left \{ \begin{pmatrix} 1 & 0 \\ y &1  \end{pmatrix} |  \; y \in \EE \right\}. \]

Given the abstract involution $\theta : \SL(2,\EE)  \to \SL(2,\EE)$, we consider  the action of $\SL(2,\EE)$ on $\End(\EE^2)$ given by $g\cdot_{\theta} D := \rho(g) D \rho(\theta(g^{-1}))$, for every $g \in \SL(2,\EE)$ and $D \in \End(\EE^2)$. The stabilizer in $\SL(2,\EE)$, with respect to the action $\cdot_\theta$, of the element $\Id_{\EE^2}$ is the subgroup $\SL(2,\FF)$. Then $\SL(2,\EE) \cdot_{\theta} \Id_{\EE^2} \cong \SL(2,\EE)/\SL(2,\FF)$ and we consider the continuous map 
$$\psi_{\theta} : \SL(2,\EE) \to \P(\End(\EE^2)), \; \text{ given by } g \mapsto \psi_{\theta}(g):=[\rho(g)\rho(\theta(g^{-1}))]=[\rho(g \theta(g^{-1}))].$$

We want to understand the following closure with respect to the involution $\theta$ and the map $\psi_{\theta}$:
$$X_\theta:= \overline{\psi_{\theta}(\SL(2,\EE))}=\overline{[\SL(2,\EE) \cdot_{\theta} \Id_{\EE^2}]}= \overline{[\SL(2,\EE)/\SL(2,\FF)]} \text{ in }\P(\End(\EE^2)).$$

\begin{definition}
We regard the above closure $X_\theta$ as a compactification of $ \SL(2,\EE)/\SL(2,\FF)$ and call it the \textit{the wonderful compactification of $\SL(2,\EE)/\SL(2,\FF)$ with respect to the involution $\theta$}.
\end{definition}

Just from the definition, we have the following trivial lemma:

\begin{lemma}
\label{lem::dense_orbit_E}
The $\SL(2,\EE)$-orbit of $\Id_{\EE^2}$  in $\P(\End(\EE^2))$, with respect to the action $\cdot_\theta$,  is dense in  $X_{\theta}$.
\end{lemma}

In order to compute the wonderful compactification of $\SL(2,\EE)/\SL(2,\FF)$ we first need to find the number of orbits of $\SL(2,\FF)$ on the boundary $\partial \tri_{\EE}$.

\begin{lemma}
\label{lem::nr_orbits_SL(F)_SL(E)}
There are at most $5$ $\SL(2,\FF)$-orbits on the boundary $\partial \tri_{\EE}\cong P^1\EE=\EE \cup \{\infty\}$: 
\begin{enumerate}
\item 
the $\SL(2,\FF)$-orbit of  $\big[ \begin{smallmatrix} 1 \\ 0 \end{smallmatrix} \big]$ which is closed
\item
for each $m \in \FF^{*}/(\FF^{*})^{2}$, the $\SL(2,\FF)$-orbit of  $\big[ \begin{smallmatrix} 1 \\ m\alpha \end{smallmatrix} \big]$ is open and might coincide with  the orbit of a different $m$.
\end{enumerate}
\end{lemma}

\begin{proof}
It is clear that the $\SL(2,\FF)$-orbit of  $\big [ \begin{smallmatrix} 1 \\ 0 \end{smallmatrix} \big]$ is closed in $\partial \tri_{\EE}$ with respect to the cone topology on $\partial \tri_{\EE}$ (see \cite{FTN} page 4).

Consider the subgroup  $B:=\{ \left( \begin{smallmatrix} d^{-1} &0 \\ c & d\end{smallmatrix} \right) \; \vert \; c \in \FF, \; d \in \FF^{*}\}$ in $\SL(2,\FF)$. Then the image of $\big[\begin{smallmatrix}  1 \\ m\alpha \end{smallmatrix} \big] $ under $B$ is $\big[\begin{smallmatrix} 1 \\ cd + d^{2}m\alpha \end{smallmatrix} \big] $, which covers the entire $\partial \tri_{\EE}- \partial \tri_\FF$ when $m$ takes all values from $\FF^{*}/(\FF^{*})^{2}$. Moreover, for each $m \in \FF^{*}/(\FF^{*})^{2}$, it is easy to notice that the set $\{ \big[\begin{smallmatrix} 1 \\ cd + d^{2}m\alpha \end{smallmatrix} \big] \; \vert\; c \in \mathcal{O}_\FF, d \in \mathcal{O}_\FF^{*}\}$ is open in $\partial \tri_{\EE}$ with respect to the cone topology on $\partial \tri_{\EE}$. Thus the $\SL(2,\FF)$-orbit of  $\big[ \begin{smallmatrix} 1 \\ m\alpha \end{smallmatrix} \big]$ is open, for each $m \in \FF^{*}/(\FF^{*})^{2}$.
\end{proof}


Consider the matrices in $\SL(2,\EE)=\SL(2,\FF(\alpha))$
$$g_1:=\Id_{\EE^2},  \qquad g_m:= \left(\begin{smallmatrix} 1 & -\frac{1}{m\alpha} \\ 0 &1 \end{smallmatrix} \right) \text{ for } m \in \FF^{*}/(\FF^{*})^{2}$$
such that 
 $$g_m\left(\big[ \begin{smallmatrix} 1 \\ m\alpha \end{smallmatrix} \big]\right)= \big[ \begin{smallmatrix} 0 \\1 \end{smallmatrix} \big].$$

Let $I :=\{ \Id_{\FF^2}, g_m \; \vert \; m \in \FF^{*}/(\FF^{*})^{2} \}$, and $I_m :=\{g_m \; \vert \; m \in \FF^{*}/(\FF^{*})^{2} \}$.
\begin{lemma}
\label{lem::decomp_E_SL}
With the above notation we have that:
\begin{enumerate}
\item
$\SL(2,\EE)=  \bigsqcup\limits_{g_i \in I} B^{-}_{\EE}g_i \SL(2,\FF)$
\item
for any $g_m \in I_m$,  the set $B^{-}_{\EE}g_m \SL(2,\FF) = U^{-}_{\EE}T_{\EE} g_m \SL(2,\FF)$ is open in $\SL(2,\EE)$, and moreover the set $\bigsqcup\limits_{g_m \in I_m} U^{-}_{\EE}T_{\EE}g_m  \SL(2,\FF)$ is dense in $\SL(2,\EE)$
\item
$\psi_{\theta}(\bigsqcup\limits_{g_m \in I_m} U^{-}_{\EE}T_{\EE} g_m )$ is open and dense in $\psi_{\theta}(\SL(2,\EE))=[\SL(2,\EE)  \cdot_{\theta} \Id_{\EE^2}]$, and so  $$\overline{\psi_{\theta}(\bigsqcup\limits_{g_m \in I_m} U^{-}_{\EE} T_{\EE} g_m)}=\overline{[\SL(2,\EE) \cdot_{\theta} \Id_{\EE^2}]}.$$
\end{enumerate}
\end{lemma}
\begin{proof}
Let $g \in \SL(2,\EE)$ and let $\xi_0 := \big[  \begin{smallmatrix} 0 \\ 1 \end{smallmatrix} \big]$. The $\SL(2,\EE)$-stabilizer of $\xi_0$ is  $B^{-}_{\EE}$.

If $g^{-1}(\xi_0)$ is in the $\SL(2,\FF)$-orbit of $\xi_0$ then $g\in B^{-}_{\EE} \SL(2,\FF)$. If not then there are $4$ cases to consider, each corresponding to the open $\SL(2,\FF)$-orbits in $\partial \tri_{\EE}$ from Lemma \ref{lem::nr_orbits_SL(F)_SL(E)}. By applying accordingly an element $h \in \SL(2,\FF)$ we get that $hg^{-1}(\xi_0) \in \{\big[  \begin{smallmatrix} 1 \\ m\alpha  \end{smallmatrix} \big] \;\vert \; \overline{m} \in \FF^{*}/(\FF^{*})^{2} \}$. Further, by multiplying again accordingly with the element $g_m$ from $I_m$ we obtain $g_m hg^{-1}(\xi_0) = \xi_0$, and thus $g^{-1}\in \SL(2,\FF) g_m^{-1}B^{-}_{\FF}$, giving $g \in  B^{-}_{\FF}g_m \SL(2,\FF)$ as required. This proves part (1) of the lemma. 

Let us prove part (2). Take $g_m \in I_m$. By Lemma \ref{lem::nr_orbits_SL(F)_SL(E)} the  $\SL(2,\FF)$-orbit of  $\big[ \begin{smallmatrix} 1 \\ m\alpha \end{smallmatrix} \big]= g_m^{-1}\left(\big[ \begin{smallmatrix} 0 \\ 1 \end{smallmatrix} \big]\right)$ is open in $\partial \tri_{\EE}$. Because the cone topology on $\SL(2,\EE)/B^{-}_{\EE} \cong \partial \tri_{\EE}$ is equivalent to the quotient topology on $\SL(2,\EE)/B^{-}_{\EE}$ induced from the continuous and open canonical projection of $\SL(2,\EE)$ to  $\SL(2,\EE)/B^{-}_{\EE}$, we have $ \SL(2,\FF) g_{m}^{-1} B^{-}_{\EE}$ is open in $\SL(2,\EE)$, and by taking the inverse map which is continuous, the set $B^{-}_{\EE}g_m\SL(2,\FF)$ is open in $\SL(2,\EE)$ as well.

It remains to prove that  $\bigsqcup\limits_{g_m \in I_m} U^{-}_{\EE}T_{\EE}g_m \SL(2,\FF)$ is dense in $\SL(2,\FF)$, and by part (1) it is enough to show that for any fixed $g_m \in I_m$, the accumulation points of the open set $B^{-}_{\EE}g_m \SL(2,\FF)$ are exactly the elements of $B^{-}_{\EE}\SL(2,\FF)$. Using again the cone topology on $\SL(2,\EE)/B^{-}_{\EE} \cong \partial \tri_{\EE}$ and the group $B$ from the proof of Lemma \ref{lem::nr_orbits_SL(F)_SL(E)} one easily gets part (2) of the lemma.

Part (3) follows easily from part (2), the fact that $B^{-}_{\EE}=U^{-}_{\EE}T_{\EE}$, and the continuity of the map $ \psi_{\theta}$.
\end{proof}

In this section we will consider the set $$\P_{0}:=\{x =[x_1, 1; x_3, x_4] \in \P(\End(\EE^2)) \}$$ which is open in $\P(\End(\EE^2))$. Then we denote $$X_{\theta, 0}:= \P_{0} \cap X_{\theta}$$ and this is open in $X_\theta$. As before, the set  $X_{\theta, 0}$ is called the big cell in  $X_\theta$. 

By Lemma \ref{lem::decomp_E_SL} it is again advised to study the sets $(U^{-}_{\EE}T_{\EE}g_i)\cdot_{\theta}\Id_{\EE^{2}}$, for any $g_i \in I$, as well as their closure.

We have
\begin{equation}
\label{equ::14}
\begin{split}
(U^{-}_{\EE}T_{\EE})\cdot_{\theta}\Id_{\EE^{2}}&= \left \{ \begin{pmatrix} x & 0 \\ a & x^{-1} \end{pmatrix}  \theta\left(\begin{pmatrix} x^{-1} & 0 \\ -a & x  \end{pmatrix} \right)  \; | \; x \in \EE^*, a\in\EE  \right\}\\
&= \left \{ \begin{pmatrix} x & 0 \\ a & x^{-1}  \end{pmatrix} \begin{pmatrix}\theta(x^{-1}) & 0 \\ \theta(-a) &\theta(x)  \end{pmatrix} \;| \; x \in \EE^*, a\in\EE  \right\}\\
&= \left \{ \begin{pmatrix} x\theta(x^{-1}) & 0 \\ a\theta(x^{-1}) + x^{-1}\theta(-a) & x^{-1}\theta(x)  \end{pmatrix} \;|\; x \in \EE^*, a\in\EE  \right\} \Rightarrow \\
& \Rightarrow [ (U^{-}_{\EE}T_{\EE})\cdot_{\theta}\Id_{\EE^{2}}] = \{[ x\theta(x^{-1}), 0; a\theta(x^{-1}) + x^{-1}\theta(-a),  x^{-1}\theta(x) ] \;| \; x \in \EE^*, a \in\EE \}. \\
 \end{split}
\end{equation}

For any $g_m \in I_m$ we also have:
\begin{equation}
\label{equ::24}
\begin{split}
(U^{-}_{\EE}T_{\EE} g_m)\cdot_{\theta}\Id_{\EE^{2}}&= \left \{ \begin{pmatrix} x & 0 \\ a & x^{-1} \end{pmatrix}  \left(\begin{matrix} 1 & -\frac{1}{m\alpha} \\ 0 &1 \end{matrix} \right)  \theta\left(  \left(\begin{matrix} 1 & \frac{1}{m\alpha} \\ 0 &1 \end{matrix} \right)\begin{pmatrix} x^{-1} & 0 \\ -a & x  \end{pmatrix} \right)  \; | \; x \in \EE^*, a\in\EE  \right\}\\
&= \left\{ \begin{pmatrix} x & -\frac{x}{m\alpha} \\ a & -\frac{a}{m\alpha} + x^{-1}  \end{pmatrix} \begin{pmatrix}\theta(x^{-1}) + \frac{\theta(a)}{m\alpha} & -\frac{\theta(x)}{m\alpha} \\ -\theta(a) &\theta(x)  \end{pmatrix} \;| \; x \in \EE^*, a\in\EE  \right\}\\
& \xRightarrow[x^{-1}\theta(x^{-1})]{multiply}\left\{ \left[\begin{pmatrix} 1 & -\frac{1}{m\alpha} \\ \frac{a}{x} & -\frac{a}{mx\alpha} + x^{-2}  \end{pmatrix} \begin{pmatrix}\theta(x^{-2}) + \frac{\theta(a)}{m\theta(x)\alpha} & -\frac{1}{m\alpha} \\ -\frac{\theta(a)}{\theta(x)} &1 \end{pmatrix} \right] \;| \; x \in \EE^*, a\in\EE  \right\}=\\
&
= [ (U^{-}_{\EE}T_{\EE})\cdot_{\theta}\Id_{\EE^{2}}] \subset \P_{0}. \\
 \end{split}
\end{equation}
Then the only accumulation points of the set $[ (U^{-}_{\EE}T_{\EE}g_m)\cdot_{\theta}\Id_{\EE^{2}}]$ are the points $[-\theta(b),1;-b\theta(b),b]$ with $b \in \EE$. Those elements appear when taking the sequences $\{x_{n_k}:=\omega_{\EE}^{-n_k}\}_{k\geq 0}$ and $\{a_{n_k}:=bx_{n_k}\}_{k\geq 0}$, with $ 0< n_k \to \infty$ and $b \in \EE$. For sequences of the form $\{x_{n_k}:=\omega_{\EE}^{n_k}\}_{k\geq 0}$ we obtain no limits.  As well, notice that the set $\{[-\theta(b),1;-b\theta(b),b]\; | \; b\in \EE \}$ is the $U^{-}_{\EE}T_{\EE}$-orbit of the element $[0,1;0,0] \in \P(\End(\EE^2))$ with respect to the action $\cdot_{\theta}$.

As a result of Lemma \ref{lem::decomp_E_SL} and the above calculations in (\ref{equ::14}) and (\ref{equ::24}) we get:
\begin{lemma}
\label{lem::X0_int_G_H_E}
Recall the notation $I_m :=\{g_m \; \vert \; m \in \EE^{*}/(\EE^{*})^{2} \}$. Then:
$$X_{\theta,0} \cap \psi_{\theta}(\SL(2,\EE))= [(\bigsqcup\limits_{g_m \in I_m} U^{-}_{\EE}T_{\EE}g_m)\cdot_{\theta}\Id_{\EE^{2}}]$$
$$X_{\theta,0} =\overline{[ (\bigsqcup\limits_{g_m \in I_m} U^{-}_{\EE}T_{\EE}g_m)\cdot_{\theta}\Id_{\EE^{2}}]} \cap  \P_{0}$$
$$X_{\theta}= \bigcup\limits_{g \in \SL(2,\EE)} g\cdot_{\theta}X_{\theta,0}.$$
\end{lemma}
\begin{proof}
It is clear from the above calculation that
$$ [(\bigsqcup\limits_{g_m \in I_m} U^{-}_{\EE}T_{\EE}g_m)\cdot_{\theta}\Id_{\EE^{2}}] = X_{\theta,0} \cap \psi_{\theta}(\SL(2,\EE))=  \P_{0} \cap \psi_{\theta}(\SL(2,\EE)). $$
Since $\overline{[ (\bigsqcup\limits_{g_m \in I_m} U^{-}_{\EE}T_{\EE}g_m)\cdot_{\theta}\Id_{\EE^{2}}]} =  X_{\theta}$ by  Lemma \ref{lem::decomp_E_SL}(3), we immediately have 
$$X_{\theta,0} =\overline{[ (\bigsqcup\limits_{g_m \in I_m} U^{-}_{\EE}T_{\EE}g_m)\cdot_{\theta}\Id_{\EE^{2}}]} \cap  \P_{0}.$$

Let $C \in X_{\theta}$. As $C \in \P(\End(\EE^2))$ we have $C = [x_1, x_2; x_3,x_4]$. If $x_2 \neq 0$ then $C \in X_{\theta,0}$ and we are done. If $x_1 =0$, then by some easy calculations one can arrange for some $g \in \SL(2,\EE)$, such that $\rho(g) C \rho(\theta(g^{-1})) \in \P_0$. Then we are done.
\end{proof}

We are now ready to prove the main result of this section.

\begin{proof}[Proof of Theorem \ref{thm::all_orbits_E}]
By Lemma \ref{lem::X0_int_G_H_E}, the $\SL(2,\EE)$-orbits in $X_{\theta}$ with respect to the $\cdot_{\theta}$ action are determined by the $\SL(2,\EE)$-orbits given by the points of the set $X_{\theta,0}$. Moreover, the calculations in (\ref{equ::24}) show the only accumulation points of the sets $[ (U^{-}_{\EE}T_{\EE}g_m)\cdot_{\theta}\Id_{\EE^{2}}]$  in $\P_0$ are the points $[-\theta(b),1;-b\theta(b),b]$ with $b \in \EE$,  which is the $U^{-}_{\EE}T_{\EE}$-orbit of the element $[0,1;0,0] \in \P(\End(\EE^2))$ with respect to the action $\cdot_{\theta}$. In addition, it is very easy to see that the only accumulation point of the set $[T_{\EE}\cdot_{\theta}\Id_{\EE^{2}}]$ in $\P_0$ is just $[0,1;0,0]$.

Further, by an easy computation 
$$ \begin{pmatrix} a & b \\ c & d \end{pmatrix} \begin{pmatrix} 0 & 1 \\ 0 & 0 \end{pmatrix} \begin{pmatrix}\theta(d) & -\theta(b) \\ -\theta(c) & \theta(a)  \end{pmatrix}= \begin{pmatrix} -a\theta(c) & a\theta(a) \\ -c\theta(c) & c\theta(a) \end{pmatrix}$$
showing that the $\SL(2,\EE)$-orbit of the point $[0,1;0,0]$ is indeed closed in $X_{\theta}$.

As $[(\bigsqcup\limits_{g_m \in I_m} U^{-}_{\EE}T_{\EE}g_m)\cdot_{\theta}\Id_{\EE^{2}}]$ is already part of the $\SL(2,\EE)$-orbit of the point $[1,0;0,1] = [\Id_{\EE^{2}}]$ with respect to the $\cdot_{\theta}$ action, the theorem follows.
\end{proof}

\begin{corollary}
With respect to the action $\cdot_{\theta}$, the stabilizer in $\SL(2,\EE)$ of the point $[1,0;0,1] =[\Id_{\EE^{2}}] $ is $\SL(2,\FF)$, and the stabilizer in $\SL(2,\EE)$ of the point $[0,1;0,0]$ is the subgroup 
$\{\left( \begin{smallmatrix} a-\alpha b& z \\  0 &  a+\alpha b \end{smallmatrix} \right) \; \vert \;  a,b \in\FF \text{ with }a^2 - \alpha^2 b^{2}=1,  z \in \EE \} \leq B_{\EE}^{+}.$ 
\end{corollary}
\begin{proof}
The second part just follows from the computation 
$$ \begin{pmatrix} a & b \\ c & d \end{pmatrix} \begin{pmatrix} 0 & 1 \\ 0 & 0 \end{pmatrix} \begin{pmatrix}\theta(d) & -\theta(b) \\ -\theta(c) & \theta(a)  \end{pmatrix}= \begin{pmatrix} -a\theta(c) & a\theta(a) \\ -c\theta(c) & c\theta(a) \end{pmatrix} =  \begin{pmatrix} 0 & 1 \\ 0 & 0 \end{pmatrix}.$$
\end{proof}

\section{The wonderful compactification of $\SL(2,\FF)/H_{\sigma}$}
\label{sec::all_orbits}

The third wonderful compactification is given by the family of inner involutions $\sigma$ of $\SL(2,\FF)$ with fixed point groups $H_{\sigma}=\{g \in \SL(2,\FF) \vert \; \sigma(g)=g\}$.  In this section we will use the notation from (\ref{equ::notations}).

In the next few paragraphs we summarize the corresponding results from \cite{HW_class} for $k$-involutions of $\SL(2,k)$ when $k$ is a field of characteristic not $2$. Let $\overline{k}$ be the algebraic closure of $k$.

Recall, a mapping $\phi: \SL(2,\overline{k}) \to \SL(2,\overline{k})$ is a \emph{$k$-automorphism} (or equivalently,  \emph{an automorphism defined over $k$}) if $\phi$ is a bijective rational k-homomorphism whose inverse is also a rational $k$-homomorphism,  \cite[Sec. 2.2]{Hel_k_invol}. An abstract automorphism $\theta$ of $\SL(2,\overline{k})$ of order two is called an \emph{abstract involution} of $\SL(2,\overline{k})$. A \emph{$k$-involution} $\theta$ of $\SL(2,\overline{k})$  is an involution defined over $k$ of $\SL(2,\overline{k})$, and the restriction of $\theta$ to $\SL(2,k)$  is a \emph{$k$-involution} of $\SL(2,k)$.  Given $g \in \SL(2,\overline{k})$ denote by $\iota_g$ the inner automorphism of $\SL(2,\overline{k})$  defined by $x \mapsto \iota_g(x):= gxg^{-1}$.

The classification of the isomorphism classes of $k$-involutions of  a connected reductive algebraic group defined over $k$ is given in \cite{Hel_k_invol}. A simple characterization of the isomorphism classes of $k$-involutions of $\SL(n,k)$ is given in \cite{HelWD}.
We record the classification of $k$-involutions of $\SL(2,k)$:

 \begin{theorem}\label{class_invol}[\cite{HW_class} Theorem 1, Corollary 1, Corollary 2]. Every $k$-isomorphism class of $k$-involution of $\SL(2,k)$ is of the form $\iota_A$ with $A= \left( \begin{smallmatrix} 0 &1 \\ m &0 \end{smallmatrix} \right) \in \GL(2,k)$. Two such $k$-involutions $\iota_A$ with $ A\in \{ \left( \begin{smallmatrix} 0 &1 \\ m &0 \end{smallmatrix} \right),  \left( \begin{smallmatrix} 0 &1 \\ m' &0 \end{smallmatrix} \right)\} \subset \GL(2,k)$ of $\SL(2,k)$ are conjugate if and only if $m$ and $m'$ are in the same square class of $k^{*}$. In particular, there are $order(k^* / (k^ *)^2)$ $k$-isomorphism classes of $k$-involutions of $\SL(2,k)$. 
 \end{theorem}

\begin{definition} Given an involution $\sigma$ of a group $G$ the \emph{fixed point group} of $\sigma$ is $H_{\sigma}: = \{ x \in G \; \vert \; \sigma(x) = x \}$. 
\end{definition} 

For $\sigma$ a $k$-involution of $\SL(2,k)$ the quotient $\SL(2,k)/H_{\sigma}$ is called a \emph{$k$-symmetric variety}, and much of the structure of $\SL(2,k)/H_{\sigma}$ is determined by $H_\sigma$. 

\begin{proposition}[\cite{HW_class} Section 3]
\label{prop::prop_shape_inv}
 Let $\sigma= \iota_A$, with $A= \left( \begin{smallmatrix} 0 &1 \\ m &0 \end{smallmatrix} \right)  \in \GL(2,k)$, be a $k$-involution of $\SL(2,k)$. Then $ H_{\sigma}=\{  \left( \begin{smallmatrix} x &y \\ my &x\end{smallmatrix} \right) \in \SL(2,k)\; \vert \; x^2 - my^2 =1 \} $. 
\end{proposition}

\begin{theorem} [\cite{HW_class} Section 3.2]
\label{thm::thm_anist_iso} Let $k=\QQ_p$, $\sigma = \iota_A$ with $A= \left( \begin{smallmatrix} 0 &1 \\ m &0 \end{smallmatrix} \right)$ and $\overline{m} \in \QQ_{p}^{*} / (\QQ_p^*)^2$.  Then $H_\sigma$ is anisotropic if and only if $\bar m \neq \bar 1$.  If $\bar m = \bar 1$, then $H_\sigma$ is isotropic and conjugate to the maximal $\QQ_p$-split torus of $\SL(2,\QQ_p)$, i.e. the diagonal subgroup of $\SL(2,\QQ_p)$. 
\end{theorem} 

\begin{theorem}[\cite{Beun} Theorem 4.18]
\label{thm::orbits_H}
Let $k=\QQ_p$, $\sigma = \iota_A$ with $A= \left( \begin{smallmatrix} 0 &1 \\ m &0 \end{smallmatrix} \right)$ and $\overline{m} \in \QQ_{p}^{*} / (\QQ_p^*)^2$. Then 
\begin{enumerate}
\item
$\vert B^{+} \backslash \SL(2,\QQ_p)/H_{\sigma} \vert=2$ if  $\overline{m} \neq \overline{1}$
\item
$\vert B^{+} \backslash \SL(2,\QQ_p)/H_{\sigma} \vert=6$ if  $\overline{m} = \overline{1}$.
\end{enumerate}
\end{theorem}

The same results as in Theorem \ref{thm::orbits_H} should hold true for any finite field-extension $\FF$ of $\QQ_p$. Below we will give a purely geometric proof of such results.

We also record the following geometric interpretation of the fixed point group $H_{\sigma}$, when $k=\FF$ is a finite field-extension of $\QQ_p$. To fix notation, let $\FF$ be a finite field-extension of $\QQ_p$. We denote by $\tri_\FF$ the Bruhat--Tits tree for $\SL(2,\FF)$ whose vertices  are equivalence classes of $\mathcal{O}_\FF$-lattices in $\FF^2$ (for its construction see \cite{Serre}). The tree $\tri_\FF$ is a regular, infinite tree with valence $\vert k_\FF\vert +1$ at every vertex, where $ k_\FF$ is the residue field of $\FF$.  The boundary at infinity $\partial \tri_{\FF}$ of $\tri_{\FF}$ is the projective space $P^1(\FF) \cong\FF \cup \{\infty\}$. Moreover, the endpoint $\infty \in \partial \tri_{\FF}$ corresponds to the vector  $\big[ \begin{smallmatrix}  0 \\1 \end{smallmatrix} \big] \in P^1(\FF)$. The rest of the endpoints $\xi \in \partial \tri_\FF$ correspond to the vectors $\big[ \begin{smallmatrix} 1 \\x \end{smallmatrix} \big] \in P^1(\FF) $, where $x \in \FF$. 

\begin{theorem}[\cite{CL_2} Corollary 4.8]
\label{lem::geom_k_inv_SL_2}
Let $\FF$ be a finite field-extension of $\QQ_p$, $A=\left( \begin{smallmatrix} 0 &1 \\ m &0 \end{smallmatrix} \right)$, with $m\in \FF^{*}/(\FF^{*})^{2}$, and $\sigma:=\iota_A$ the corresponding $\FF$-involution of $\SL(2,\FF)$.  Take $K_m:=\FF(\sqrt{m})$  a field extension. 
Then the only solutions of the equation $A(\xi)=\xi$ with $\xi \in P^{1}K_a$ are   $\xi_{\pm} := \big[ \begin{smallmatrix} 1 \\ \pm \sqrt{m} \end{smallmatrix} \big]$ and $H_{\sigma} = \Fix_{\SL(2,K_m)}(\{\xi_{-},\xi_{+}\}) \cap \SL(2,\FF) =\{  \left( \begin{smallmatrix} x &y \\ my &x\end{smallmatrix} \right) \in \SL(2,\FF)\; \vert \; x^2 - my^2 =1 \} $. Moreover,
\begin{enumerate}
\item
if $m=1$ then  $\xi_{\pm} :=\big[ \begin{smallmatrix}  1 \\ \pm 1 \end{smallmatrix} \big]$  and  $H_{\sigma}$ contains all the hyperbolic elements of $\SL(2,\FF)$ with $\xi_{\pm}$ as their repelling and attracting endpoints. In particular, $H_{\sigma}$ is  $\GL(2,\FF)$-conjugate to the entire diagonal subgroup of $\SL(2,\FF)$,  thus $H_{\sigma}$ is noncompact and abelian. 
\item
if $m\neq 1$ then  $\xi_{\pm} :=\big[  \begin{smallmatrix}  1 \\ \pm  \sqrt{m}  \end{smallmatrix} \big]  \in P^{1}K_m - P^{1}\FF$, and $H_{\sigma}$ is compact and abelian.
\end{enumerate}
\end{theorem}

Now, given an involution $\sigma: = \iota_A$ as in Theorem \ref{class_invol}, we consider the following action of $\SL(2,\FF)$ on $\End(\FF^2)$ given by $g\cdot_{\sigma} D := \rho(g) D \rho(\sigma(g^{-1}))$, for every $g \in \SL(2,\FF)$ and $D \in \End(\FF^2)$. The stabilizer in $\SL(2,\FF)$, with respect to the action $\cdot_\sigma$, of the element $\Id_{\FF^2}$ is the subgroup $H_{\sigma}$. Then $\SL(2,\FF) \cdot_{\sigma} \Id_{\FF^2} \cong \SL(2,\FF)/H_{\sigma}$ and we consider the continuous map 
$$\psi_{\sigma} : \SL(2,\FF) \to \P(\End(\FF^2)), \; \text{ given by } g \mapsto \psi_{\sigma}(g):=[\rho(g)\rho(\sigma(g^{-1}))]=[\rho(g \sigma(g^{-1}))].$$

We want to understand the following closure with respect to the involution $\sigma$ and the map $\psi_{\sigma}$:
$$X_\sigma:= \overline{\psi_{\sigma}(\SL(2,\FF))}=\overline{[\SL(2,\FF) \cdot_{\sigma} \Id_{\FF^2}]}= \overline{[\SL(2,\FF)/H_{\sigma}]} \text{ in }\P(\End(\FF^2)).$$

\begin{definition}
We regard the above closure $X_\sigma$ as a compactification of $ \SL(2,\FF)/H_{\sigma}$ and call it the \textit{the wonderful compactification of $\SL(2,\FF)/H_{\sigma}$ with respect to the involution $\sigma$}.
\end{definition}

Just from the definition, we have the following trivial lemma:

\begin{lemma}
\label{lem::dense_orbit}
The $\SL(2,\FF)$-orbit of $\Id_{\FF^2}$  in $\P(\End(\FF^2))$, with respect to the action $\cdot_\sigma$,  is dense in  $X_{\sigma}$.
\end{lemma}


In order to prove Theorem \ref{thm::all_orbits} we will again need the set $$\P_{0}:=\{x =[1,  x_2; x_3, x_4] \in \P(\End(\FF^2)) \}$$ which is open in $\P(\End(\FF^2))$. Then denote $$X_{\sigma, 0}:= \P_{0} \cap X_{\sigma}$$ and this is open in $X_\sigma$. The set  $X_{\sigma, 0}$ is also called the big cell in  $X_\sigma$.

Instead of using the density of $U^{-}T \times U$ in $\SL(2,\FF)$ as in Section \ref{sec::diag} it will be much more convenient to employ the fact that the double coset $B^{+} \backslash \SL(2,\FF)/H_{\sigma}$ has a finite number of elements. Below we will explicitly compute those double coset for each $\overline{m} \in \FF^{*} / (\FF^*)^2$, and find the double cosets that are open in $\SL(2,\FF)$. Along the way we will apply the geometric picture provided by Theorem \ref{lem::geom_k_inv_SL_2}. Since the constuction below works by replacing the Borel $B^{+}$ with any of its $\SL(2,\FF)$-conjugates, we will use the opposite $B^{-}$.

Recall that for any locally compact group $G$ and any closed subgroup $H \leq G$, the quotient topology on the homogeneous space $G/H$ is defined by the canonical projection $p : G \to G/H$ being continuous and open. Moreover, one can prove that for any compact subset $Q$ of $G/H$, there exists a compact subset $K$ of $G$ with $p(K) = Q$ (see \cite[Appendices B, Lemma B.1.1]{BHV}). 

The cases $\overline{m} = \overline{1}$ and $\overline{m} \neq \overline{1}$  behave differently, since for the later we will use a quadratic extension of $\FF$.

\textbf{Case $\overline{m} = \overline{1}$.}

Recall from \cite{CL_2} the following trivial lemma.
\begin{lemma}
\label{lem::nr_orbits_diag_SL}
Let $K$ be a finite field-extension of $\QQ_p$. There are $6$ orbits of the diagonal subgroup  $Diag(K):=\{ \left( \begin{smallmatrix} d^{-1} &0 \\ 0 & d\end{smallmatrix} \right) \; \vert \; d \in K^{*}\} \leq \SL(2,K)$ on the boundary $\partial \tri_K$:
\begin{enumerate}
\item 
the $Diag(K)$-orbit of  $\big[\begin{smallmatrix} 1 \\ 0 \end{smallmatrix} \big]$ and the $Diag(K)$-orbit of  \big[$\begin{smallmatrix} 0 \\ 1 \end{smallmatrix} \big] $
\item
the $Diag(K)$-orbits of  $\big[\begin{smallmatrix} 1 \\ \overline{m} \end{smallmatrix} \big]$, for each $\overline{m} \in K^{*}/(K^{*})^{2}$. 
\end{enumerate}
\end{lemma}
\begin{proof}
The subgroup $Diag(K)$ fixes pointwise the ends  $\big[ \begin{smallmatrix} 1 \\ 0 \end{smallmatrix} \big]$ and $\big[\begin{smallmatrix} 0 \\ 1 \end{smallmatrix} \big]$.   The $Diag(K)$-orbit of $\big[ \begin{smallmatrix}  1 \\ \overline{m} \end{smallmatrix} \big]$, for each  $\overline{m} \in K^{*}/(K^{*})^{2}$, consists of vectors of the form $\big[ \begin{smallmatrix} 1 \\ d^{2} \overline{m} \end{smallmatrix} \big]$, these cover the entire boundary   $\partial \tri_{E} - \{\big[ \begin{smallmatrix}  0 \\ 1 \end{smallmatrix} \big] ,\big[  \begin{smallmatrix} 1 \\ 0 \end{smallmatrix} \big]\} $.
\end{proof}

As a trivial consequence of Lemma \ref{lem::nr_orbits_diag_SL} one obtains:
\begin{lemma}
\label{lem::nr_orbits_H_SL}
Let $\FF$ be a finite field-extension of $\QQ_p$, $A=\left( \begin{smallmatrix} 0 &1 \\ 1 &0 \end{smallmatrix} \right)$, and $\sigma_1:=\iota_A$ the corresponding $\FF$-involution of $\SL(2,\FF)$ with associated fixed point group $H_{\sigma_1}$. Then
$$H_{\sigma_1} = \left(\begin{smallmatrix} 1 &-1 \\ 1 &1 \end{smallmatrix} \right) Diag(\FF) \left(\begin{smallmatrix} 1 &-1 \\ 1 &1 \end{smallmatrix} \right)^{-1}, $$
$H_{\sigma_1}$ has exaclty $6$ orbits on the boundary $\partial \tri_{\FF}$, and the corresponding $H_{\sigma_1}$-orbits on the boundary $\partial \tri_{\FF}$ are given by the following $6$ representatives: $\{\big[  \begin{smallmatrix} 1-\overline{m} \\ 1+\overline{m} \end{smallmatrix} \big] \;\vert \; \overline{m} \in \FF^{*}/(\FF^{*})^{2} \} \cup \{ \big[  \begin{smallmatrix} 1 \\ -1 \end{smallmatrix} \big],\big[  \begin{smallmatrix} 1 \\ 1 \end{smallmatrix} \big]\} \subset \partial \tri_{\FF}$.
\end{lemma}

Let be the matrices in $\SL(2,\FF)$
$$g_1:=\left(\begin{smallmatrix} -1 &-1 \\ 2 &1 \end{smallmatrix} \right),  \qquad g_2:= \left(\begin{smallmatrix} -1 &-1 \\ 1 &0 \end{smallmatrix} \right), \qquad g_m:= \left(\begin{smallmatrix} 1+m & m-1 \\ \frac{m}{m-1} &1 \end{smallmatrix} \right) \text{ for } m \in \FF^{*}/(\FF^{*})^{2} \text{ with } m\neq 1$$
such that 
$$g_1\left(\big[ \begin{smallmatrix} 1 \\ -1 \end{smallmatrix} \big]\right)= \big[ \begin{smallmatrix} 0 \\ 1 \end{smallmatrix} \big], \qquad g_2\left(\big[ \begin{smallmatrix} 1 \\ 1 \end{smallmatrix} \big]\right)= \big[ \begin{smallmatrix} 0 \\ 1 \end{smallmatrix} \big], \qquad g_m\left(\big[ \begin{smallmatrix} 1-m \\ 1+m \end{smallmatrix} \big]\right)= \big[ \begin{smallmatrix} 0 \\ 1 \end{smallmatrix} \big].$$

Let $I :=\{ \Id_{\FF^2}, g_1,g_2, g_m \; \vert \; m \in \FF^{*}/(\FF^{*})^{2}, m\neq 1 \}$, and $I_m :=\{g_m \; \vert \; m \in \FF^{*}/(\FF^{*})^{2}, m\neq 1 \} \cup \{\Id_{\FF^2}\}$, with the convention that for $m = 1 \in \FF^{*}/(\FF^{*})^{2}$ we take $g_m := \Id_{\FF^2}$.

\begin{lemma}
\label{lem::decomp_H_SL}
With the above notation we have that:
\begin{enumerate}
\item
$\SL(2,\FF)=  \bigsqcup\limits_{g_i \in I} B^{-}g_i H_{\sigma_1}$
\item
for any $g_m \in I_m$,  the set $B^{-}g_m H_{\sigma_1} = U^{-}Tg_m  H_{\sigma_1}$ is open in $\SL(2,\FF)$, and moreover the set $\bigsqcup\limits_{g_m \in I_m} U^{-}Tg_m H_{\sigma_1}$ is dense in $\SL(2,\FF)$
\item
$\psi_{\sigma_1}(\bigsqcup\limits_{g_m \in I_m} U^{-}Tg_m )$ is open and dense in $\psi_{\sigma_1}(\SL(2,\FF))=[\SL(2,\FF)  \cdot_{\sigma_1} \Id_{\FF^2}]$, and so  $$\overline{\psi_{\sigma_1}(\bigsqcup\limits_{g_m \in I_m} U^{-}Tg_m)}=\overline{[\SL(2,\FF) \cdot_{\sigma_1} \Id_{\FF^2}]}.$$
\end{enumerate}
\end{lemma}
\begin{proof}
Let $g \in \SL(2,\FF)$ and let $\xi_0 := \big[  \begin{smallmatrix} 0 \\ 1 \end{smallmatrix} \big]$. Notice the $\SL(2,\FF)$-stabilizer of $\xi_0$ is exactly $B^{-} =\left \{ \begin{pmatrix} x & 0 \\ y &x^{-1}  \end{pmatrix} |  \; y \in \FF, x \in \FF^*  \right\}$.

If $g^{-1}(\xi_0)=\xi_0$ then $g\in B^{-}$ and thus $g \in  B^{-} H_{\sigma_1}$. If $g^{-1}(\xi_0) \neq \xi_0$ then there are $6$ cases to consider each corresponding to the $H_{\sigma_1}$-orbits in $\partial \tri_{\FF}$. By applying accordingly an element $h \in H_{\sigma_1}$ we get that $hg^{-1}(\xi_0) \in \{\big[  \begin{smallmatrix} 1-\overline{m} \\ 1+\overline{m} \end{smallmatrix} \big] \;\vert \; \overline{m} \in \FF^{*}/(\FF^{*})^{2} \} \cup \{ \big[  \begin{smallmatrix} 1 \\ -1 \end{smallmatrix} \big],\big[  \begin{smallmatrix} 1 \\ 1 \end{smallmatrix} \big]\}$. Further, by multiplying again accordingly with some element $g_i$ from $I$ we obtain $g_i hg^{-1}(\xi_0) = \xi_0$, and thus $g^{-1}\in  H_{\sigma_1} g_i^{-1}B^{-}$, giving $g \in  B^{-}g_i H_{\sigma_1}$ as required. This proves part (1) of the lemma. 

Let us prove part (2). Take $g_m \in I_m$. Since by Lemma \ref{lem::nr_orbits_H_SL} the group $H_{\sigma_1}$ is conjugate to $Diag(\FF)$, and since the $Diag(\FF)$-orbit of the vector $\big[ \begin{smallmatrix} 1 \\ \overline{m} \end{smallmatrix} \big] \in \partial \tri_{\FF}$,  where  $\overline{m} \in \FF^{*}/(\FF^{*})^{2}$, is clearly open with respect to the cone topology on $\SL(2,\FF)/B^{-} \cong \partial \tri_{\FF}$, then the $H_{\sigma_1}$-orbit of $\big[  \begin{smallmatrix} 1-\overline{m} \\ 1+\overline{m} \end{smallmatrix} \big] \in \partial \tri_{\FF}$ is open in $\partial \tri_{\FF}$. Because the cone topology on $\SL(2,\FF)/B^{-} \cong \partial \tri_{\FF}$ is equivalent to the quotient topology on $\SL(2,\FF)/B^{-}$ induced from the continuous and open canonical projection of $\SL(2,\FF)$ to  $\SL(2,\FF)/B^{-}$, we have $H_{\sigma_1} g_{m}^{-1} B^{-}$ is open in $\SL(2,\FF)$, and by taking the inverse map which is continuous, the set $B^{-}g_m H_{\sigma_1}$ is open in $\SL(2,\FF)$ as well.

It remains to prove that  $\bigsqcup\limits_{g_m \in I_m} U^{-}Tg_m H_{\sigma_1}$ is dense in $\SL(2,\FF)$, and by part (1) it is enough to show that for any fixed $g_m \in I_m$, the accumulation points of the open set $B^{-}g_m H_{\sigma_1}$ are exactly the elements of $B^{-}g_1 H_{\sigma_1} \cup B^{-}g_2 H_{\sigma_1}$. Indeed, since $H_{\sigma_1}$ is conjugate to the diagonal subgroup $Diag(\FF)$, its attacting and repealing endpoints are $\{ \big[  \begin{smallmatrix} 1 \\ -1 \end{smallmatrix} \big],\big[  \begin{smallmatrix} 1 \\ 1 \end{smallmatrix} \big]\}$. Notice that for any fixed $g_m \in I_m$, the open set  $H_{\sigma_1}\big[  \begin{smallmatrix} 1-\overline{m} \\ 1+\overline{m} \end{smallmatrix} \big]$ has $\{ \big[  \begin{smallmatrix} 1 \\ -1 \end{smallmatrix} \big],\big[  \begin{smallmatrix} 1 \\ 1 \end{smallmatrix} \big]\}$ as its unique accumulation points in $\partial \tri_{\FF}$. Using again the cone topology on $\SL(2,\FF)/B^{-} \cong \partial \tri_{\FF}$, there exist sequences $\{b_n\}_{n\geq 1} \subset B^{-}$ and $\{h_n\}_{n\geq 1} \subset H_{\sigma_1}$ such that $\{b_n g_m h_n \}_{n\geq 1}$ converges to $g_1$, or $g_2$, with respect to the topology on $\SL(2,\FF)$. Then part (2) follows.

Part (3) follows easily from part (2), the fact that $B^{-}=U^{-}T$, and the continuity of the map $ \psi_{\sigma_1}$.
\end{proof}

By Lemma \ref{lem::decomp_H_SL} it is advised to study the sets $(U^{-}Tg_i)\cdot_{\sigma_1}\Id_{\FF^{2}}$, for any $g_i \in I$, as well as their closure.

We have
\begin{equation}
\label{equ::1}
\begin{split}
(U^{-}Tg_1 )\cdot_{\sigma_1}\Id_{\FF^{2}}&= \left \{ \begin{pmatrix} x & 0 \\ a & x^{-1} \end{pmatrix} \begin{pmatrix} -1 &-1 \\ 2 &1 \end{pmatrix} \begin{pmatrix} 0 & 1 \\ 1 &0  \end{pmatrix} \begin{pmatrix} 1 &1 \\ -2 &-1 \end{pmatrix}  \begin{pmatrix} x^{-1} & 0 \\ -a & x  \end{pmatrix} \begin{pmatrix} 0 & 1 \\ 1 &0  \end{pmatrix} \; | \; x \in \FF^*, a\in\FF  \right\}\\
&= \left \{ \begin{pmatrix} x & 0 \\ a & x^{-1}  \end{pmatrix} \begin{pmatrix} 1 &0 \\ -3 & -1 \end{pmatrix} \begin{pmatrix} 0 & x^{-1} \\ x &-a  \end{pmatrix} \;| \; x \in \FF^*, a\in\FF  \right\}\\
&= \left \{ \begin{pmatrix} 0 & 1 \\ -1 &2 a x^{-1}-3x^{-2}  \end{pmatrix} \;|\; x \in \FF^*, a\in\FF  \right\} \Rightarrow \\
& \Rightarrow [ (U^{-}T g_1)\cdot_{\sigma_1}\Id_{\FF^{2}}] = \{[0, 1; -1, 2 a x^{-1}-3x^{-2}] \;| \; x \in \FF^*, a \in\FF \} \subset [\SL(2,\FF) \cdot_{\sigma_1} \Id_{\FF^2}]. \\
 \end{split}
\end{equation}

\begin{equation}
\label{equ::2}
\begin{split}
(U^{-}Tg_2 )\cdot_{\sigma_1}\Id_{\FF^{2}}&= \left \{ \begin{pmatrix} x & 0 \\ a & x^{-1} \end{pmatrix} \begin{pmatrix} -1 &-1 \\ 1 &0 \end{pmatrix} \begin{pmatrix} 0 & 1 \\ 1 &0  \end{pmatrix} \begin{pmatrix} 0&1 \\ -1 &-1 \end{pmatrix}  \begin{pmatrix} x^{-1} & 0 \\ -a & x  \end{pmatrix} \begin{pmatrix} 0 & 1 \\ 1 &0  \end{pmatrix} \; | \; x \in \FF^*, a\in\FF  \right\}\\
&= \left \{ \begin{pmatrix} x & 0 \\ a & x^{-1}  \end{pmatrix} \begin{pmatrix} 1 &0 \\ -1 & -1 \end{pmatrix} \begin{pmatrix} 0 & x^{-1} \\ x &-a  \end{pmatrix} \;| \; x \in \FF^*, a\in\FF  \right\}\\
&= \left \{ \begin{pmatrix} 0 & 1 \\ -1 &2 a x^{-1}-x^{-2}  \end{pmatrix} \;|\; x \in \FF^*, a\in\FF  \right\} \Rightarrow \\
& \Rightarrow [ (U^{-}T g_2 )\cdot_{\sigma_1}\Id_{\FF^{2}}] = \{[0, 1; -1,  a x^{-1}-x^{-2}] \;| \; x \in \FF^*, a \in\FF \}  \subset [\SL(2,\FF) \cdot_{\sigma_1} \Id_{\FF^2}]. \\
\end{split}
\end{equation}

Now for every $m \in \FF^{*}/(\FF^{*})^{2}$, with $m\neq 1$, we have:
\begin{equation}
\label{equ::3}
\begin{split}
&(U^{-}Tg_m )\cdot_{\sigma_1}\Id_{\FF^{2}}= \left \{ \begin{pmatrix} x & 0 \\ a & x^{-1} \end{pmatrix} \begin{pmatrix} 1+m &m-1 \\ \frac{m}{m-1} &1 \end{pmatrix} \begin{pmatrix} 0 & 1 \\ 1 &0  \end{pmatrix} \begin{pmatrix} 1&1-m \\  \frac{m}{1-m} &1+m \end{pmatrix}  \begin{pmatrix} x^{-1} & 0 \\ -a & x  \end{pmatrix} \begin{pmatrix} 0 & 1 \\ 1 &0  \end{pmatrix} \; | \; x \in \FF^*, a\in\FF  \right\}\\
&= \left \{ \begin{pmatrix} x & 0 \\ a & x^{-1}  \end{pmatrix} \begin{pmatrix} \frac{3m-1}{1-m} &4m \\ \frac{1-2m}{(1-m)^{2}} & \frac{3m-1}{m-1} \end{pmatrix} \begin{pmatrix} 0 & x^{-1} \\ x &-a  \end{pmatrix} \;| \; x \in \FF^*, a\in\FF  \right\}\\
&= \left \{ \begin{pmatrix} 4mx^2 & \frac{3m-1}{1-m} - 4max\\ 4max + \frac{3m-1}{m-1} &2 a x^{-1} \frac{3m-1}{1-m} +x^{-2} \frac{1-2m}{(1-m)^2}-4ma^2  \end{pmatrix} \;|\; x \in \FF^*, a\in\FF  \right\} \Rightarrow \\
& \Rightarrow [ (U^{-}T g_m )\cdot_{\sigma_1}\Id_{\FF^{2}}]=\\
& = \{[1, \frac{3m-1}{(1-m)4mx^2} -ax^{-1}; ax^{-1}+\frac{3m-1}{(m-1)4mx^2},  \frac{2a}{4mx^{2}} \frac{3m-1}{1-m}+ \frac{1-2m}{(1-m)^2 4mx^4}- a^2x^{-2}] \;| \; x \in \FF^*, a \in\FF \}. \\
\end{split}
\end{equation}
Then the only accumulation points of the set $[ (U^{-}T g_m)\cdot_{\sigma_1}\Id_{\FF^{2}}]$ are the points $[1,-b;b,-b^2]$ with $b \in \FF$. Those elements appear when taking  sequences $\{x_{n_k}=\omega_{\FF}^{-n_k}\}_{k\geq 0}$ and $\{a_{n_k}:=bx_{n_k}\}_{k\geq 0}$,  with $ 0< n_k \to \infty$ and $b \in \FF$. For sequences of the form $\{x_{n_k}=\omega_{\FF}^{n_k}\}_{k\geq 0}$ we obtain no limits.

For the case $m \in \FF^{*}/(\FF^{*})^{2}$ with $m=1$, by our convention above, we take $g_m:=\Id_{\FF^2}$ and we have:
\begin{equation}
\label{equ::4}
\begin{split}
(U^{-}T )\cdot_{\sigma_1}\Id_{\FF^{2}}&= \left \{ \begin{pmatrix} x & 0 \\ a & x^{-1}  \end{pmatrix}\begin{pmatrix} 0 & 1 \\ 1 &0  \end{pmatrix}  \begin{pmatrix} x^{-1} & 0 \\ -a & x  \end{pmatrix} \begin{pmatrix} 0 & 1 \\ 1 &0  \end{pmatrix} \; | \; x \in \FF^*, a\in\FF  \right\}\\
&= \left \{ \begin{pmatrix} 0 & x \\ x^{-1} & a  \end{pmatrix} \begin{pmatrix} 0 & x^{-1} \\ x &-a  \end{pmatrix} \;| \; x \in \FF^*, a\in\FF  \right\}\\
& =\left \{ \begin{pmatrix} x^{2} & -xa \\ ax & -a^{2}+x^{-2}  \end{pmatrix} \;|\; x \in \FF^*, a\in\FF  \right\} \Rightarrow \\
& \Rightarrow [ (U^{-}T )\cdot_{\sigma_1}\Id_{\FF^{2}}] = \{[1,-x^{-1}a;ax^{-1}, -a^{2} x^{-2}+x^{-4}] \;| \; x \in \FF^*, a \in\FF \}. \\
\end{split}
\end{equation}
Then the only accumulation points of the set $[ (U^{-}T g_m)\cdot_{\sigma_1}\Id_{\FF^{2}}]$ are the points $[1,-b;b,-b^2]$ with $b \in \FF$. Those elements appear when taking  sequences $\{x_{n_k}=\omega_{\FF}^{-n_k}\}_{k\geq 0}$ and $\{a_{n_k}:=bx_{n_k}\}_{k\geq 0}$,  with $ 0< n_k \to \infty$ and $b \in \FF$. For sequences of the form $\{x_{n_k}=\omega_{\FF}^{n_k}\}_{k\geq 0}$ we obtain no limits. 

 As well, notice that the set $\{[1,-b;b,-b^2]\; | \; b\in \FF \}$ is the $U^{-}T$-orbit of the element $[1,0;0,0] \in \P(\End(\FF^2))$ with respect to the action $\cdot_{\sigma_1}$.

As a result of Lemma \ref{lem::decomp_H_SL} and the above calculations in (\ref{equ::1}) to (\ref{equ::4}) we have:
\begin{lemma}
\label{lem::X0_int_G_H}
Recall the notation $I_m :=\{g_m \; \vert \; m \in \FF^{*}/(\FF^{*})^{2}, m\neq 1 \} \cup \{\Id_{\FF^2}\}$. Then:
$$X_{\sigma_1,0} \cap \psi_{\sigma_1}(\SL(2,\FF))= [(\bigsqcup\limits_{g_m \in I_m} U^{-}Tg_m)\cdot_{\sigma_1}\Id_{\FF^{2}}]$$
$$X_{\sigma_1,0} =\overline{[ (\bigsqcup\limits_{g_m \in I_m} U^{-}Tg_m)\cdot_{\sigma_1}\Id_{\FF^{2}}]} \cap  \P_{0}$$
$$X_{\sigma_1}= \bigcup\limits_{g \in \SL(2,\FF)} g\cdot_{\sigma_1}X_{\sigma_1,0}.$$
\end{lemma}
\begin{proof}
It is clear from the above calculation that
$$ [(\bigsqcup\limits_{g_m \in I_m} U^{-}Tg_m)\cdot_{\sigma_1}\Id_{\FF^{2}}] = X_{\sigma_1,0} \cap \psi_{\sigma_1}(\SL(2,\FF))=  \P_{0} \cap \psi_{\sigma_1}(\SL(2,\FF)). $$
Since $\overline{[ (\bigsqcup\limits_{g_m \in I_m} U^{-}Tg_m)\cdot_{\sigma_1}\Id_{\FF^{2}}]} =  X_{\sigma_1}$ by  Lemma \ref{lem::decomp_H_SL}(3), we immediately have 
$$X_{\sigma_1,0} =\overline{[ (\bigsqcup\limits_{g_m \in I_m} U^{-}Tg_m)\cdot_{\sigma_1}\Id_{\FF^{2}}]} \cap  \P_{0}.$$

Let $C \in X_{\sigma_1}$. As $C \in \P(\End(\FF^2))$ we have $C = [x_1, x_2; x_3,x_4]$. If $x_1 \neq 0$ then $C \in X_{\sigma_1,0}$ and we are done. If $x_1 =0$, then by some easy calculations one can arrange for some $g \in \SL(2,\FF)$, such that $\rho(g) C \rho(\sigma_1(g^{-1})) \in \P_0$. Then we are done.
\end{proof}

\begin{proof}[Proof of Theorem \ref{thm::all_orbits}]
By Lemma \ref{lem::X0_int_G_H}, the $\SL(2,\FF)$-orbits in $X_{\sigma_1}$ with respect to the $\cdot_{\sigma_1}$ action are determined by the $\SL(2,\FF)$-orbits given by the points of the set $X_{\sigma_1,0}$. Moreover, the calculations in (\ref{equ::3}) and (\ref{equ::4}) show the only accumulation points of the set $[ (U^{-}T g_m)\cdot_{\sigma_1}\Id_{\FF^{2}}]$ in $\P_0$ are the points $[1,-b;b,-b^2]$ with $b \in \FF$, and the only accumulation point of the set $[T\cdot_{\sigma_1}\Id_{\FF^{2}}]$ in $\P_0$ is $[1,0;0,0]$.  As well, the set $\{[1,-b;b,-b^2]\; | \; b\in \FF \}$ is the $U^{-}T$-orbit of the element $[1,0;0,0] \in \P(\End(\FF^2))$ with respect to the action $\cdot_{\sigma_1}$. Further, an easy computation 
$$\begin{pmatrix} a & b \\ c & d \end{pmatrix} \begin{pmatrix} 1 & 0 \\ 0 & 0 \end{pmatrix} \begin{pmatrix} 0 & 1 \\ 1 & 0 \end{pmatrix} \begin{pmatrix} d & -b \\ -c & a \end{pmatrix} \begin{pmatrix} 0 & 1 \\ 1 & 0 \end{pmatrix} = \begin{pmatrix} a & b \\ c & d \end{pmatrix} \begin{pmatrix} 1 & 0 \\ 0 & 0 \end{pmatrix} \begin{pmatrix}a & -c \\ -b & d \end{pmatrix}= \begin{pmatrix} a^{2} & -ac \\ ca & -c^{2} \end{pmatrix}$$
shows that the $\SL(2,\FF)$-orbit of the point $[1,0;0,0]$ is indeed closed in $X_{\sigma_1}$.

As $[(\bigsqcup\limits_{g_m \in I_m} U^{-}Tg_m)\cdot_{\sigma_1}\Id_{\FF^{2}}]$ is already part of the $\SL(2,\FF)$-orbit of the point $[1,0;0,1] = [\Id_{\FF^{2}}]$ with respect to the $\cdot_{\sigma_1}$ action, the theorem follows.
\end{proof}

\begin{corollary}
\label{cor::1_case}
With respect to the action $\cdot_{\sigma_1}$, the stabilizer in $\SL(2,\FF)$ of the point $[1,0;0,1]=[\Id_{\FF^{2}}]$ is $H_{\sigma_1}$, and the stabilizer in $\SL(2,\FF)$ of the point $[1,0;0,0]$ is  the subgroup $\{ \mu _2 \cdot \left( \begin{smallmatrix} 1& b \\ 0 &1  \end{smallmatrix} \right)\; |\; b \in \FF  \}$ of the Borel  $B^{+}\leq \SL(2,\FF)$, where $\mu _2$ is the group of $2^{nd}$ roots of unity in  $\FF$.
\end{corollary}
\begin{proof}
This just follows from the computation
$$ \begin{pmatrix} a & b \\ c & d \end{pmatrix} \begin{pmatrix} 1 & 0 \\ 0 & 0 \end{pmatrix} \begin{pmatrix}a & -c \\ -b & d \end{pmatrix}= \begin{pmatrix} a^{2} & -ac \\ ca & -c^{2} \end{pmatrix}.$$
\end{proof}


\textbf{Case $\overline{m} \neq \overline{1}$.}

Recall for $m\in \FF^{*}/(\FF^{*})^{2}$ with $ \overline{m} \neq \overline{1}$ we consider $A:=\left( \begin{smallmatrix} 0 &1 \\ m &0 \end{smallmatrix} \right)$ which has associated the $\FF$-involution of $\SL(2,\FF)$ given by $\sigma_m:=\iota_A$.

From \cite{CL_2} also recall Lemma 5.3:
\begin{lemma}
\label{lem::nr_orbits_H_SL_m}
Let $\FF$ be a finite field extension of $\QQ_p$, $A=\left( \begin{smallmatrix} 0 &1 \\ m &0 \end{smallmatrix} \right)$, with $m\in \FF^{*}/(\FF^{*})^{2}$, $m\neq 1$ and $\sigma_m:=\iota_A$ the corresponding $\FF$-involution of $\SL(2,\FF)$.  
Then $H_{\sigma_m}:=\{g \in \SL(2,\FF) \; \vert \; \sigma_m(g) = g\}$ has at most $8$ orbits on the boundary $\partial T_\FF$.
\end{lemma}

Choose now a set $I \subset \SL(2,\FF)$ of representatives such that the $H_{\sigma_m}$-orbits of $g_i^{-1}\left(\big[ \begin{smallmatrix} 0 \\ 1 \end{smallmatrix} \big]\right) \in \partial T_{\FF}$, for $g_i \in I$, are all disjoint and cover the boundary $\partial T_{\FF}$. By Lemma \ref{lem::nr_orbits_H_SL_m} we know that $I$ is a finite set. 

\begin{lemma}
\label{lem::decomp_H_SL_m}
With the above notation we have that:
\begin{enumerate}
\item
$\SL(2,\FF)=  \bigsqcup\limits_{g_i \in I} B^{-}g_i H_{\sigma_m}$
\item
for any $g_i \in I$,  the set $B^{-}g_i H_{\sigma_m} = U^{-}Tg_i H_{\sigma_m}$ is open in $\SL(2,\FF)$
\item
$\psi_{\sigma_m}(\bigsqcup\limits_{g_i \in I} U^{-}Tg_i )$ is open and dense in $\psi_{\sigma_m}(\SL(2,\FF))=[\SL(2,\FF)  \cdot_{\sigma_m} \Id_{\FF^2}]$, and so  $$\overline{\psi_{\sigma_m}(\bigsqcup\limits_{g_i \in I} U^{-}Tg_i)}=\overline{[\SL(2,\FF) \cdot_{\sigma_m} \Id_{\FF^2}]}.$$
\end{enumerate}
\end{lemma}
\begin{proof}
Let $g \in \SL(2,\FF)$ and let $\xi_0 := \big[  \begin{smallmatrix} 0 \\ 1 \end{smallmatrix} \big]$. Notice the $\SL(2,\FF)$-stabilizer of $\xi_0$ is exactly $B^{-} =\left \{ \begin{pmatrix} x & 0 \\ y &x^{-1}  \end{pmatrix} |  \; y \in \FF, x \in \FF^*  \right\}$.

If $g^{-1}(\xi_0)=\xi_0$ then $g\in B^{-}$ and thus $g \in  B^{-} H_{\sigma_m}$. If $g^{-1}(\xi_0) \neq \xi_0$ then there are exactly $\vert I \vert < \infty$ cases to consider, each corresponding to the $H_{\sigma_m}$-orbits in $\partial \tri_{\FF}$. By applying accordingly an element $h \in H_{\sigma_m}$ we get that $hg^{-1}(\xi_0) \in \{ g_i^{-1}(\xi_0) \;\vert \; g_i \in I \}$. Further, by multiplying again accordingly with some element $g_i$ from $I$ we obtain $g_i hg^{-1}(\xi_0) = \xi_0$, and thus $g^{-1}\in  H_{\sigma_m} g_i^{-1}B^{-}$, giving $g \in  B^{-}g_i H_{\sigma_m}$ as required. This proves part (1) of the lemma. 

Let us prove part (2). Consider the quadratic field extention $\EE:=\FF(\sqrt{m})$. By Theorem \ref{lem::geom_k_inv_SL_2}(2) we know that the two ends $\xi{\pm}:=\big[  \begin{smallmatrix}  1 \\ \pm  \sqrt{m}  \end{smallmatrix} \big]$ are in $P^{1}\EE - P^{1}\FF = \partial \tri_\EE - \partial \tri_\FF$. Moreover, the group $\SL(2,\FF)$ acts on $\tri_{\EE}$, it preserves the subset of ends $\partial \tri_\FF$, and in fact the subset $\partial \tri_\FF$ is closed in $\partial \tri_\EE$ with respect to the cone topology on $\partial \tri_\EE$. Now, again by Theorem \ref{lem::geom_k_inv_SL_2} we know that $\Fix_{\SL(2,\EE)}(\{\xi_{-},\xi_{+}\})$ is a conjugate of the diagonal $Diag(\EE) \leq \SL(2,\EE)$, and by the same proof as in Lemma \ref{lem::decomp_H_SL}(2),  $\Fix_{\SL(2,\EE)}(\{\xi_{-},\xi_{+}\})$ has $4$ open orbits on $\partial \tri_\EE - \{\xi_{\pm}\}$. In particular, this means the open orbits of $\Fix_{\SL(2,\EE)}(\{\xi_{-},\xi_{+}\})$ cover $\partial \tri_\FF$, and the intersection of an open orbit of   $\Fix_{\SL(2,\EE)}(\{\xi_{-},\xi_{+}\})$ in $\partial \tri_\EE$ with the closed set $\partial \tri_\FF$ remains open with respect to the cone topology on $\partial \tri_\FF$ and $\tri_\FF$. In addition, by the proof of Lemma \ref{lem::nr_orbits_H_SL_m} from \cite{CL_2}[ Lemma 5.3] we know that $H_{\sigma_m}=\Fix_{\SL(2,\EE)}(\{\xi_{-},\xi_{+}\}) \cap \SL(2,\FF)$, has at most two orbits on any of the open orbits of $\Fix_{\SL(2,\EE)}(\{\xi_{-},\xi_{+}\})$ interesecting $\partial \tri_\FF$. Since $H_{\sigma_m}$ is compact and the action of $\SL(2,\FF)$ on $\partial \tri_\FF$ is continuous, any $H_{\sigma_m}$-orbit on $\partial \tri_\FF$ is open with respect to the cone topololgy on $\partial T_\FF$.  Because the cone topology on $\SL(2,\FF)/B^{-} \cong \partial \tri_{\FF}$ is equivalent to the quotient topology on $\SL(2,\FF)/B^{-}$ induced from the continuous and open canonical projection of $\SL(2,\FF)$ to  $\SL(2,\FF)/B^{-}$, we have $H_{\sigma_m} g_{i}^{-1} B^{-}$ is open in $\SL(2,\FF)$, and by taking the inverse map which is continuous, the set $B^{-}g_i H_{\sigma_1}$ is open in $\SL(2,\FF)$ as well.

Part (3) easily follows from parts (1) and (2).
\end{proof}

Take $g= \begin{pmatrix} a & b \\ c & d \end{pmatrix} \in \SL(2,\FF)$ then we have:
\begin{equation}
\label{equ::31}
\begin{split}
[(U^{-}T g)\cdot_{\sigma_m}\Id_{\FF^{2}}]&= \left \{ \left[\begin{pmatrix} x & 0 \\ y & x^{-1}  \end{pmatrix}  \begin{pmatrix} a & b \\ c & d \end{pmatrix}\begin{pmatrix} 0 & 1 \\ m &0  \end{pmatrix}  \begin{pmatrix} d & -b \\ -c & a \end{pmatrix}  \begin{pmatrix} x^{-1} & 0 \\ -y & x  \end{pmatrix} \begin{pmatrix} 0 & \frac{1}{m} \\ 1 &0  \end{pmatrix}\right] \; | \; x \in \FF^*, y\in\FF  \right\}\\
&= \left \{ \left[\begin{pmatrix} xa & xb \\ ya+x^{-1}c & yb+x^{-1}d \end{pmatrix} \begin{pmatrix} -c & a \\ md&-mb  \end{pmatrix}  \begin{pmatrix} 0 & \frac{1}{mx} \\ x& \frac{-y}{m}  \end{pmatrix}\right] \;| \; x \in \FF^*, y\in\FF  \right\}\\
& =\left \{ \left[\begin{pmatrix} xa & xb\\ ya + \frac{c}{x} & yb+\frac{d}{x}  \end{pmatrix} \begin{pmatrix} xa & -\frac{c}{mx} -\frac{ya}{m}\\ -mxb & \frac{d}{x} + by  \end{pmatrix}\right] \;|\; x \in \FF^*, y\in\FF  \right\} \\
&\xlongequal[x^{-2}]{multiply} \left \{ \left[\begin{pmatrix} a & b\\ \frac{ya}{x} + \frac{c}{x^2} & \frac{yb}{x}+\frac{d}{x^2}  \end{pmatrix} \begin{pmatrix} a & -\frac{c}{mx^2} -\frac{ya}{mx}\\ -mb & \frac{d}{x^2} + \frac{by}{x}  \end{pmatrix}\right] \;|\; x \in \FF^*, y\in\FF  \right\} \\
\end{split}
\end{equation}
Then the only accumulation points of the set $[ (U^{-}T g )\cdot_{\sigma_m}\Id_{\FF^{2}}] \subset  \P_{0}$ are the points 
$$\left\{ \left[ \begin{pmatrix} a & b\\ za & zb  \end{pmatrix} \begin{pmatrix} a &-\frac{za}{m}\\ -mb & bz  \end{pmatrix}\right]=[1,-\frac{z}{m};z,-\frac{z^2}{m}] \; \vert \;   z\in \FF\right\}$$ since we always have $a^2 -mb^2 \neq 0$.  Those elements appear when taking sequences $\{x_{n_k}=\omega_{\FF}^{-n_k}\}_{k\geq 0}$ and $\{y_{n_k}:=zx_{n_k}\}_{k\geq 0}$,  with $ 0< n_k \to \infty$ and $z \in \FF$. For sequences of the form $\{x_{n_k}=\omega_{\FF}^{n_k}\}_{k\geq 0}$ we obtain no limits. 

As well, notice that the set $\{[1,-\frac{z}{m};z,-\frac{z^2}{m}]\; | \; z\in \FF \}$ is the $U^{-}T$-orbit of the element $[1,0;0,0] \in \P(\End(\FF^2))$ with respect to the action $\cdot_{\sigma_m}$.

\begin{remark}
\label{rem::easy_rem}
 We combine Lemma \ref{lem::decomp_H_SL_m}(1) and (2) with the calculation (\ref{equ::31}). Therefore, since  $\psi_{\sigma_m}(\bigsqcup\limits_{g_i \in I} U^{-}Tg_i ) \subset \P_0$ is open and equals $\psi_{\sigma_m}(\SL(2,\FF))=[\SL(2,\FF)  \cdot_{\sigma_m} \Id_{\FF^2}]$ we notice that the big cell $X_{\sigma, 0}:= \P_{0} \cap X_{\sigma_m}$ in $X_{\sigma_m}$ is exactly the set $X_{\sigma_m}$.

\end{remark}

\begin{proof}[Proof of Theorem \ref{thm::all_orbits}]
By Remark \ref{rem::easy_rem}, the $\SL(2,\FF)$-orbits in $X_{\sigma_m}$ with respect to the $\cdot_{\sigma_m}$ action are determined by the $\SL(2,\FF)$-orbits given by the points of the set $X_{\sigma_m,0}$. Moreover, the calculations in (\ref{equ::31}) show the only acumulation points of $[(\bigsqcup\limits_{g \in I} U^{-}Tg)\cdot_{\sigma_m}\Id_{\FF^{2}}]$  in $\P_0$ are the points $[1,-\frac{z}{m};z,-\frac{z^2}{m}]$, with $z \in \FF$, which are exactly the $U^{-}T$-orbit of the element $[1,0;0,0] \in \P(\End(\FF^2))$ with respect to the action $\cdot_{\sigma_m}$.   In addition, the only accumulation point of the set $[T\cdot_{\sigma_m}\Id_{\FF^{2}}]$ in $\P_0$ is $[1,0;0,0]$.

Further, by an easy computation 
$$\begin{pmatrix} a & b \\ c & d \end{pmatrix} \begin{pmatrix} 1 & 0 \\ 0 & 0 \end{pmatrix} \begin{pmatrix} 0 & 1 \\ m & 0 \end{pmatrix} \begin{pmatrix} d & -b \\ -c & a \end{pmatrix} \begin{pmatrix} 0 & \frac{1}{m} \\ 1 & 0 \end{pmatrix} = \begin{pmatrix} a & b \\ c & d \end{pmatrix} \begin{pmatrix} 1 & 0 \\ 0 & 0 \end{pmatrix} \begin{pmatrix}a & \frac{-c}{m} \\ -bm & d \end{pmatrix}= \begin{pmatrix} a^{2} & \frac{-ac}{m} \\ ca & \frac{-c^{2}}{m} \end{pmatrix}.$$
showing that the $\SL(2,\FF)$-orbit of the point $[1,0;0,0]$ is indeed closed in $X_{\sigma_m}$.

As $[(\bigsqcup\limits_{g_i \in I} U^{-}Tg_i)\cdot_{\sigma_m}\Id_{\FF^{2}}]$ is already the $\SL(2,\FF)$-orbit of the point $[1,0;0,1] = [\Id_{\FF^{2}}]$ with respect to the $\cdot_{\sigma_m}$ action, the theorem follows.
\end{proof}

\begin{corollary}
\label{cor::m_case}
With respect to the action $\cdot_{\sigma_m}$, the stabilizer in $\SL(2,\FF)$ of the point $[1,0;0,1]= [\Id_{\FF^{2}}]$ is $H_{\sigma_m}$, and the stabilizer in $\SL(2,\FF)$ of the point $[1,0;0,0]$ is  the subgroup $\{ \mu _2 \cdot \left( \begin{smallmatrix} 1& b \\ 0 &1  \end{smallmatrix} \right)\; |\; b \in \FF  \}$ of the Borel  $B^{+}\leq \SL(2,\FF)$, where $\mu _2$ is the group of $2^{nd}$ roots of unity in  $\FF$.
\end{corollary}
\begin{proof}
This just follows from the computation
$$\begin{pmatrix} a & b \\ c & d \end{pmatrix} \begin{pmatrix} 1 & 0 \\ 0 & 0 \end{pmatrix} \begin{pmatrix} 0 & 1 \\ m & 0 \end{pmatrix} \begin{pmatrix} d & -b \\ -c & a \end{pmatrix} \begin{pmatrix} 0 & \frac{1}{m} \\ 1 & 0 \end{pmatrix} = \begin{pmatrix} a & b \\ c & d \end{pmatrix} \begin{pmatrix} 1 & 0 \\ 0 & 0 \end{pmatrix} \begin{pmatrix}a & \frac{-c}{m} \\ -bm & d \end{pmatrix}= \begin{pmatrix} a^{2} & \frac{-ac}{m} \\ ca & \frac{-c^{2}}{m} \end{pmatrix}$$
by taking $ \begin{pmatrix} a^{2} & \frac{-ac}{m} \\ ca & \frac{-c^{2}}{m} \end{pmatrix} = \begin{pmatrix} 1 & 0 \\ 0 & 0 \end{pmatrix}.$
\end{proof}



\end{document}